\documentclass[12pt]{amsart}

\usepackage[matrix,arrow,curve,frame]{xy}
\usepackage{latexsym}
\usepackage{amscd}
\usepackage{euscript}
\usepackage{tikz}
\usepackage{tikz-cd,mathtools}
\usepackage{url,pdflscape,amssymb,rotating,caption,subcaption}
\usepackage[linktocpage]{hyperref}

\usepackage{fullpage}

\usepackage{graphicx,amsmath,amsthm,amsfonts,amscd,mathrsfs} 
\theoremstyle{plain}
\newtheorem{theorem}{Theorem}[section]
\newtheorem{lemma}[theorem]{Lemma}
\newtheorem{proposition}[theorem]{Proposition}

\theoremstyle{definition}
\newtheorem{definition}[theorem]{Definition}

\theoremstyle{remark}
\newtheorem{example}{Example}[section]
\newtheorem{remark}{Remark}[section]

\newcommand\inner[2]{\left\langle #1 \left| #2 \right\rangle \right. }

\makeatletter
\newcommand*{\splitarrow}[2]{\mathrel{
  \settowidth{\@tempdima}{$\scriptstyle#1$}
  \settowidth{\@tempdimb}{$\scriptstyle#2$}
  \ifdim\@tempdimb>\@tempdima \@tempdima=\@tempdimb\fi
  \mathop{\vcenter{
    \offinterlineskip\ialign{\hbox to\dimexpr\@tempdima+1em{##}\cr
    \rightarrowfill\cr\noalign{\kern.5ex}
    \leftarrowfill\cr}}}\limits^{\!#1}_{\!#2}}}

\DeclareMathOperator{\Aut}{Aut}

\DeclareMathOperator{\Fred}{Fred}

\DeclareMathOperator{\coker}{coker}

\DeclareMathOperator{\Out}{Out}

\DeclareMathOperator{\Ext}{Ext}
\DeclareMathOperator{\Hom}{Hom}
\DeclareMathOperator{\Tor}{Tor}

\newcommand{\cH}{{\mathcal H}}

\newcommand{\cA}{{\mathcal A}}
\newcommand{\cE}{{\mathcal E}}

\newcommand{\cK}{{\mathcal K}}

\newcommand{\cZ}{{\mathcal Z}}
\newcommand{\cO}{{\mathcal O}}

\newcommand{\CC}{{\mathbb C}}
\newcommand{\RR}{{\mathbb R}}
\newcommand{\ZZ}{{\mathbb Z}}

\newcommand{\NN}{{\mathbb N}}

\newcommand{\OK}{\cO_{\infty} \otimes \cK}
\newcommand*\circles{|[draw,circle]|}

\begin{document}

\title{Computations in higher twisted $K$-theory}

\author{David Brook}
\address{School of Mathematical Sciences,
University of Adelaide, Adelaide 5005, Australia}
\email{david.brook@adelaide.edu.au}

\date{}
\thanks{The author was partially supported by a University of Adelaide Masters Scholarship}

\begin{abstract}
Higher twisted $K$-theory is an extension of twisted $K$-theory introduced by Ulrich Pennig which captures all of the homotopy-theoretic twists of topological $K$-theory in a geometric way. We give an overview of his formulation and key results, and reformulate the definition from a topological perspective. We then investigate ways of producing explicit geometric representatives of the higher twists of $K$-theory viewed as cohomology classes in special cases using the clutching construction and when the class is decomposable. Atiyah--Hirzebruch and Serre spectral sequences are developed and information on their differentials is obtained, and these along with a Mayer--Vietoris sequence in higher twisted $K$-theory are applied in order to perform computations for a variety of spaces.
\end{abstract}

\maketitle
\tableofcontents

\section{Introduction}

Higher twisted $K$-theory is the full twisted cohomology theory of topological $K$-theory, for which an operator algebraic model was first developed by Ulrich Pennig \cite{Pennig}. While twisted $K$-theory has been around for many years, it was understood that the classical twists corresponding to third-degree integral cohomology classes, or equivalently principal $PU$-bundles or bundle gerbes, did not correspond to the full set of abstract twists for $K$-theory; these twists merely had an elegant geometric interpretation. More general twists were known to exist by homotopy-theoretic arguments, but until recently these twists were entirely abstract, rendering them unusable in computations and making it difficult to find applications of these more general twists. For this reason, the limited set of twists for which geometric representatives were known have been studied extensively. In this paper, we focus on developing tools for computations involving this general class of twists, and apply these to compute the higher twisted $K$-theory groups of a wide variety of spaces.

\subsection*{Background}

What has classically been called twisted $K$-theory is a generalisation of topological $K$-theory which gradually emerged over the course of the 1960s after Atiyah and Hirzebruch's initial work in topological $K$-theory \cite{AHstart}. Interest in the area was sparked when Atiyah, Bott and Shapiro investigated the relationship between topological $K$-theory and Clifford algebras \cite{ABS}, providing a new perspective from which $K$-theory could be viewed using Clifford bundles associated to vector bundles which was studied in Karoubi's doctoral thesis \cite{KaroubiPhD}. Karoubi and Donovan extended this by replacing Clifford bundles with algebra bundles \cite{DonKar}, which resulted in what was then called ``$K$-theory with local coefficients'' in which the local coefficient systems over $X$ were classified by $H^1(X, \ZZ_2)$ and the torsion elements of $H^3(X,\ZZ)$, corresponding to finite-dimensional complex matrix algebra bundles over $X$. This was the first notion of twisted $K$-theory, and it was defined geometrically using algebra bundles to represent twists. Note that twists given by elements of $H^1(X,\ZZ_2)$ were studied comprehensively in \cite{Variant} and are often neglected in recent work.

Over the coming years, much work was done on developing twisted $K$-theory. Rosenberg extended the definition of twisted $K$-theory to any class in $H^3(X,\ZZ)$ rather than specifically considering torsion classes \cite{RosenbergTwist}, corresponding to using infinite-dimensional algebra bundles over $X$. Work by Bouwknegt, Carey, Mathai, Murray and Stevenson provided twisted $K$-theory computations and developed a twisted Chern character in the even case \cite{BCMMS}, and Mathai and Stevenson extended this to the odd case \cite{MS} and studied the Connes--Chern character for twisted $K$-theory \cite{MS2}. Further developments were later made by Atiyah and Segal in formulating twisted $K$-theory using bundles of Fredholm operators and exploring the differentials in an Atiyah--Hirzebruch spectral sequence for twisted $K$-theory \cite{AS1, AS2}.

Up until this point, twists were geometric objects which could naturally be used to modify $K$-theory. From a homotopy-theoretic point of view, however, these twists did not capture the entire picture -- there existed more general abstract twists for which no geometric interpretation was known. In fact, Atiyah and Segal acknowledged the existence of more general twists of $K$-theory which had not found a geometric realisation \cite{AS1}.

This changed with the introduction of a Dixmier--Douady theory for strongly self-absorbing $C^*$-algebras by Pennig and Dadarlat \cite{DP2,DP1}, analogous to the Dixmier--Douady theory for the compact operators on a Hilbert space which is integral to the construction of the classical twists. Pennig was able to use these algebras to propose a geometric model which captures all of the twists of $K$-theory, in which abstract twists can be realised as algebra bundles with fibres isomorphic to the stabilised infinite Cuntz algebra $\OK$ \cite{Pennig}.

\subsection*{Higher twists}
We provide a brief introduction to the homotopy-theoretic arguments showing that more general twists of $K$-theory exist. Recall that there is a notion of spectrum for a cohomology theory $h^{\bullet}$, which is a sequence of topological spaces $\{E_n\}_{n \in \NN}$ satisfying particular properties such that $h^n(X) = [X, E_n]$. In the case of topological $K$-theory, this is the 2-periodic sequence $\ZZ \times BU, U, \cdots$. To each ring spectrum there is an associated unit spectrum consisting of the units in each space, which we will denote $\{GL_1(E_n)\}_{n \in \NN}$, and from this another cohomology theory denoted $gl_1(h)^{\bullet}$ may be defined via $gl_1(h)^n(X) = [X, GL_1(E_n)]$. The notation $gl_1(h)$ here reflects that this cohomology theory is in some sense constructed out of the unital elements of the cohomology theory $h^{\bullet}$. Then the twists of $h^{\bullet}$ over some space $X$ are classified by the first group of this cohomology theory, i.e.\ $gl_1(h)^1(X)$. As $GL_1(E_0)$ denotes the units in $E_0$, we see that the twists are classified by $[X, BGL_1(E_0)]$. More technically, a twisting of a cohomology theory over a space $X$ is a bundle of spectra over $X$ with fibre given by the spectrum $R$ of the cohomology theory. Then letting $GL_1(R)$ denote the automorphism group of the spectrum $R$, these bundles of spectra are classified by $[X, BGL_1(R)]$. One may then define the groups of the twisted cohomology theory.

Applying these notions to topological $K$-theory, we see that the invertible elements of the ring $K^0(X)$ are represented by virtual line bundles. Unifying this with the spectrum picture, these classes correspond to $[X, \ZZ_2 \times BU]$ and so in the notation of the previous paragraph we have $GL_1(\ZZ \times BU) = \ZZ_2 \times BU$. Thus the twists of topological $K$-theory over $X$ are classified by \[gl_1(KU)^1(X) = [X, B(\ZZ_2 \times BU)].\] It is known that $BU$ is homotopy equivalent to $K(\ZZ,2) \times BSU$ \cite{MST} and so we see that $B(\ZZ_2 \times BU) \simeq K(\ZZ_2, 1) \times K(\ZZ,3) \times BBSU$. This means that twists of $K$-theory are classified by homotopy classes of maps
\[ X \to K(\ZZ_2,1) \times K(\ZZ,3) \times BBSU, \]
and therefore for a CW complex $X$ the twists of $K$-theory correspond to elements of $H^1(X,\ZZ_2)$, $H^3(X,\ZZ)$ and $[X,BBSU]$. The third of these groups is not well-understood, which led to the lack of understanding of this class of twists for $K$-theory. Furthermore, while this approach may be used to define what a twist of $K$-theory is, it does not provide any geometric information about twists, rendering any results specific to twisted $K$-theory very difficult to prove. Thus Pennig's geometric interpretation of the full set of twists of $K$-theory lays the foundations for significant developments in the field. We remark that while this will not be a focus of this paper, further investigation into the topology of $BBSU$ may prove useful in the study of higher twisted $K$-theory, as well as being a topic of interest in its own right.

\subsection*{Links with physics}
Further to this abstract motivation, higher twisted $K$-theory is expected to be of great relevance in mathematical physics. Both topological and classical twisted $K$-theory have found application to physics, particularly in string theory, with work of Witten \cite{Witten} providing a link between charges of $D$-branes and topological $K$-theory and later work by Bouwknegt and Mathai extending this to twisted $K$-theory in the presence of a $B$-field \cite{BM}, so it is likely that similar ties exist to higher twisted $K$-theory. In fact, some applications have already been found. In the classical case, Bouwknegt, Evslin and Mathai show that the T-duality transformation induces an isomorphism on twisted $K$-theory, which was an unexpected result as the T-duality transformation often leads to significant differences between the topologies of the circle bundles in question \cite{BEM,BEMT}. They then generalise this to the case of spherical T-duality \cite{BEM1,BEM2,BEM3} in which the relevant cohomology class is of degree 7, which led to the result that the spherical T-duality transform induces an isomorphism on higher twisted $K$-theory. In the first of this series of papers, the authors also provide some insight into how higher twisted $K$-theory fits in with the $D$-brane picture. In the setting of Type IIB string theory, the data in 10-dimensional supergravity includes a 10-manifold $Y$ which is commonly diffeomorphic to $\RR \times X$ for an appropriate 9-manifold $X$. They explain that the $K$-theory of $X$ twisted by a 7-class corresponds to the set of conserved charges of a certain subset of branes. While this is an interesting result, it only applies to 7-classes in this specialised setting which implies that there may be a richer relationship between $D$-branes and higher twisted $K$-theory than what currently exists in the literature. It may be possible to gain greater insight into the behaviour of $D$-branes by studying higher twisted $K$-theory.

Further work in studying higher twisted $K$-theory in relation to spherical T-duality is done by MacDonald, Mathai and Saratchandran \cite{MMS}, who also develop a Chern character in higher twisted $K$-theory and give an isomorphism between higher twisted $K$-theory and higher twisted cohomology over the reals.

Relevance to physics also appears in the study of strongly self-absorbing $C^*$-algebras, as some common algebras in physics such as the canonical anticommutation relations (CAR) algebra are strongly self-absorbing. This algebra is not only relevant in quantum physics, but also in the study of Clifford algebras. By exploring the version of twisted $K$-theory whose twists arise from the CAR algebra, a deeper understanding of this algebra and its applications may be gained.

A final application of twisted $K$-theory of critical importance in mathematical physics is the work of Freed, Hopkins and Teleman in proving that the Verlinde ring of positive energy representations of a loop group is isomorphic to the equivariant twisted $K$-theory of a compact Lie group \cite{FHT1, FHT3, FHT2}; a result which is of relevance in conformal field theory. It is likely that there exists an analogous result in the higher twisted setting, where it is as of yet unknown what can replace the Verlinde ring to be isomorphic to higher twisted $K$-theory, but another object from physics may arise in this case. Pennig and Evans \cite{EvansP} hint at possible approaches.

\subsection*{Main results}

The purpose of this paper is to develop computational tools in higher twisted $K$-theory and use these to perform explicit computations. The main results that we present are to achieve this goal.

One important aspect of computing higher twisted $K$-theory groups is determining explicit generators for the groups. This is difficult using the $C^*$-algebraic formulation, and so we provide an alternative definition of the higher twisted $K$-theory groups from a topological viewpoint inspired by a result of Rosenberg \cite{RosenbergTwist}. In Theorem \ref{topological} we prove that the higher twisted $K$-theory groups can be identified with
\begin{align*}
K^0(X,\delta) &= [\cE_{\delta}, \Fred_{\OK}]^{\Aut(\OK)}; \\
K^1(X,\delta) &= [\cE_{\delta}, \Omega \Fred_{\OK}]^{\Aut(\OK)};
\end{align*}
where $\cE_{\delta}$ is the principal $\Aut(\OK)$-bundle over $X$ representing the abstract twist $\delta$, $\Fred_{\OK}$ denotes the Fredholm operators on the standard Hilbert $(\OK)$-module, $\Omega$ denotes the based loop space and $[-,-]^{\Aut(\OK)} = \pi_0(C(-,-)^{\Aut(\OK)})$ denotes unbased homotopy classes of $\Aut(\OK)$-equivariant maps. This result is used to give an explicit description of the generator of $K^1(S^{2n+1},\delta)$ in Theorem \ref{generator}.

Since the definition of higher twisted $K$-theory involves sections of algebra bundles with fibre $\OK$, another task critical to computations is to explicitly construct these bundles over spaces. When the topological space is an odd-dimensional sphere $S^{2n+1}$, we note that the clutching construction can be used to construct principal $\Aut(\OK)$-bundles, or equivalently our desired algebra bundles, over $S^{2n+1}$ by specifying a gluing map $S^{2n} \to \Aut(\OK)$. These maps are classified up to homotopy by an integer, and using Pennig and Dadarlat's link between twists and cohomology classes \cite{DP1} we see that the principal $\Aut(\OK)$-bundles over $S^{2n+1}$ are also classified by an integer. This provides an explicit geometric way of describing the twist over $S^{2n+1}$ associated to a particular cohomology class; the details are presented in Section 3.1. Another case that we consider involves restricting to decomposable cohomology classes. For a general CW complex $X$ with torsion-free cohomology, we take a $5$-class $\delta$ given by the cup product of a $2$-class $\alpha$ and a $3$-class $\beta$. By associating a principal $U(1)$-bundle to $\alpha$ and a principal $PU$-bundle to $\beta$ and defining an effective action of $U(1) \times PU$ on $\OK$, we are able to construct a principal $\Aut(\OK)$-bundle over $X$ corresponding to the cohomology class $\delta$, and obtain a more general result in Theorem \ref{decomposable_theorem}. We do not apply this result in our computations, but it is likely that it can form the basis for further research.

More directly relevant to our goal, we investigate spectral sequences for higher twisted $K$-theory. We show in Theorem \ref{AHstatement} that when $X$ is a CW complex and $\delta$ is any abstract twist over $X$, there is an analogue of the Atiyah--Hirzebruch spectral sequence which has $E_2$-term $E^{p,q}_2 = H^p(X,K^q(pt))$ and which strongly converges to the higher twisted $K$-theory $K^*(X,\delta)$. In particular, when the twist $\delta$ can be identified with a cohomology class $\delta \in H^{2n+1}(X,\ZZ)$ we see in Theorem \ref{HTSS} that the $d_{2n+1}$ differential is of the form $d_{2n+1}(x) = d_{2n+1}'(x) - \delta \cup x$ where $d_{2n+1}'$ is the differential in the ordinary Atiyah--Hirzebruch spectral sequence for topological $K$-theory, which is an operator whose image is torsion \cite{Arlettaz}. There is, in fact, a more general Segal spectral sequence that can be applied in this setting and we generalise a result of Rosenberg \cite{RosenbergLie} to obtain this sequence. Letting $F \xrightarrow{\iota} E \xrightarrow{\pi} B$ be a fibre bundle of CW complexes with $\delta$ a twist over $E$, we prove in Theorem \ref{SSS} that there is a spectral sequence with $E_2$-term $E^{p,q}_2 = H^p(B,K^q(F,\iota^* \delta))$ which strongly converges to the higher twisted $K$-theory $K^*(E,\delta)$. We also obtain more explicit information about the differentials of the higher twisted $K$-homology version of this spectral sequence in Theorem \ref{differentials}, which is useful in proving general results about the higher twisted $K$-theory of Lie groups.

We conclude by using the techniques developed to perform computations. Table \ref{table} contains the explicit results obtained in Section 5, displaying the higher twisted $K$-theory groups of odd-dimensional spheres, the product of an odd-dimensional and an even-dimensional sphere, real projective spaces and lens spaces. In each case, $\delta$ is $N \neq 0$ times the generator of the cohomology group specified in the table.

\begin{table}[!htp]
  \begin{center}
    \begin{tabular}{|c|c|c|c|}
    \hline
     $X$ & Degree & $K^0(X,\delta)$ & $K^1(X,\delta)$ \\
      \hline
      $S^{2n+1}$ & $2n+1$ & $0$ & $\ZZ_N$ \\
      $S^{2m} \times S^{2n+1}$ & $2n+1$ & $0$ & $\ZZ_N \oplus \ZZ_N$ \\
      $S^{2m} \times S^{2n+1}$ & $2m+2n+1$ & $\ZZ$ & $\ZZ \oplus \ZZ_N$ \\
      $\RR P^{2n+1}$ & $2n+1$ & $\ZZ_{2^n}$ & $\ZZ_N$ \\
      $L(n,p)$ & $2n+1$ & $\ZZ_{p^n}$ & $\ZZ_N$ \\
     \hline
    \end{tabular}
\caption{Higher twisted $K$-theory groups with $\delta = N \in H^{\operatorname{degree}}(X,\ZZ) = \ZZ$}\label{table}
  \end{center}
\end{table}

We also use the Segal spectral sequence to obtain results on the higher twisted $K$-theory of $SU(n)$. In Theorem \ref{concerned} we show that for a 5-twist $\delta$ given by $N$ times the generator of $H^5(SU(n+1),\ZZ)$, with $N$ relatively prime to $n!$, the graded group $K^*(SU(n+1),\delta)$ is isomorphic to $\ZZ_N$ tensored with an exterior algebra on $n-1$ odd generators. More generally, when $\delta$ is any twist given by $N$ times a primitive generator of $H^*(SU(n),\ZZ)$ we show in Theorem \ref{sun} that $K^*(SU(n),\delta)$ is a finite abelian group with all elements having order a divisor of a power of $N$.

Our final computations are for a class of examples which arise in spherical T-duality; the total spaces of $SU(2)$-bundles. Bouwknegt, Evslin and Mathai prove that the spherical T-duality transformation induces a degree-shifting isomorphism on the 7-twisted $K$-theory groups of these bundles \cite{BEM1, BEM2, BEM3}, but they do not consider the 5-twisted $K$-theory of these $SU(2)$-bundles. We place restrictions on the base space $M$ in order to ensure that the 5-twists of the total space of the bundle correspond exactly to its integral 5-classes, and then use the Atiyah--Hirzebruch spectral sequence to compute the 5-twisted $K$-theory in Section 5.6.

\subsection*{Outline of this paper}
The paper is organised as follows. Section 2 provides the necessary background on the Cuntz algebra $\cO_{\infty}$ to formulate higher twisted $K$-theory using Pennig's original formulation, and an alternative topological characterisation is given. In Section 3, we describe methods of producing explicit geometric representatives for twists of $K$-theory using both the clutching construction and by considering decomposable cohomology classes. Section 4 develops spectral sequences for computations, and finally Section 5 contains explicit computations of higher twisted $K$-theory groups for various spaces.

\subsection*{Acknowledgements}
The author acknowledges support from the University of Adelaide in the form of an MPhil scholarship. The author would also like to thank his principal MPhil supervisor Elder Professor Mathai Varghese for suggesting the problems in this paper, explaining techniques to construct geometric representatives for twists and his excellent support throughout the project in general. The author thanks his co-supervisor Dr Peter Hochs for his exceptional guidance and support, as well as his MPhil referees and Dr David Roberts for helpful comments.

\section{Preliminaries and formulation}

\subsection{Strongly self-absorbing \texorpdfstring{$C^*$-algebras}{C*}}

We begin by presenting the background on the Cuntz algebra $\cO_{\infty}$ necessary to the formulation of higher twisted $K$-theory. We will firstly introduce Toms and Winter's class of strongly self-absorbing $C^*$-algebras, which possess a higher Dixmier--Douady theory mirroring the Dixmier--Douady theory of the compact operators and with which higher twisted $K$-theory can be defined. Apart from having this application to $K$-theory, this class of algebras is interesting in its own right as it has proved useful in the quest of Elliott to classify all simple nuclear $C^*$-algebras \cite{TWssa}. Here we present a slightly modified but equivalent definition posed by Pennig and Dadarlat \cite{DP1}, which is more applicable to topological problems. For the original definition, see \cite{TWssa}.

\begin{definition}
A seperable and unital $C^*$-algebra $D$ is called \textit{strongly self-absorbing} if there exists a $^*$-isomorphism $\psi: D \to D \otimes D$ and a path of unitaries $u: [0,1) \to U(D \otimes D)$ such that, for all $d \in D$, $\lim\limits_{t \to 1} \| \psi(d) - u_t(d \otimes 1) u_t^* \| = 0$.
\end{definition}

We are specifically concerned with explicit examples of these algebras, the most relevant of which will be the Cuntz algebra $\cO_{\infty}$ first introduced by Cuntz \cite{Cuntz}.

\begin{definition}
The \textit{Cuntz algebra} $\mathcal{O}_n$ with $n$ generators for $n =1, 2, \cdots$ is defined to be the $C^*$-algebra generated by a set of isometries $\{S_i\}_{i = 1}^{n}$ acting on a separable Hilbert space satisfying $S_i^* S_j = \delta_{ij} I$ for $i, j = 1, \cdots, n$ and
\begin{align*}
\sum\limits_{i=1}^{n} S_i S_i^* = I.
\end{align*}
Similarly, the Cuntz algebra $\mathcal{O}_{\infty}$ with infinitely many generators is defined in an analogous way for an infinite sequence $\{S_i\}_{i \in \mathbb{N}}$ satisfying $S_i^* S_j = \delta_{ij} I$ for $i, j \in \NN$ and
\[ \sum\limits_{i=1}^{k} S_i S_i^* \leq I \]
for all $k \in \NN$.
\end{definition}

It is proved in the original reference that this definition is independent of the choice of Hilbert space and of isometries. While we will not discuss twists of $K$-theory which arise from other strongly self-absorbing $C^*$-algebras in detail, these may also be of interest and so we will list other examples of strongly self-absorbing $C^*$-algebras. These algebras are shown to be strongly self-absorbing in Toms and Winter's original paper \cite{TWssa}.

\begin{example}\
\begin{enumerate}
\item[(1)] The Cuntz algebras $\cO_2$ and $\cO_{\infty}$ are strongly self-absorbing.
\item[(2)] The Jiang--Su algebra $\cZ$ introduced by Jiang and Su \cite{JiangSu} is strongly-self absorbing.
\item[(3)] Uniformly hyperfinite (UHF) algebras (defined in III.5.1 of \cite{Davidson}) of infinite type are strongly-self absorbing.
\item[(4)] The tensor product of a UHF algebra of infinite type with $\cO_{\infty}$ is strongly-self absorbing. Note that this is the only way to form a new class of algebras out of the previous examples, as $\cO_2$ absorbs UHF algebras of infinite type, all of the examples absorb $\cZ$ and $\cO_2 \otimes \cO_{\infty} \cong \cO_2$.
\end{enumerate}
\end{example}

Although all of these algebras can be used to formulate a notion of twisted $K$-theory, it is the Cuntz algebra $\cO_{\infty}$ which is used to realise the most general notion of twist, and so this is the algebra whose properties we will explore in more detail. We are particularly interested in automorphisms of $\cO_{\infty}$ and its stabilisation, so we will present some of Dadarlat and Pennig's main results on the homotopy type of these automorphism groups and explore an action of the stable unitary group on $\cO_{\infty}$ by outer automorphisms.

Firstly, a general result showing that the automorphism group of a strongly self-absorbing $C^*$-algebra does not have an interesting homotopy type.

\begin{theorem}[Theorem 2.3 \cite{DP1}]{\label{contract}}
Let $D$ be a strongly self-absorbing $C^*$-algebra. Then the space $\Aut(D)$ is contractible.
\end{theorem}

Upon stabilisation, the homotopy type of the automorphism group becomes much more interesting.

\begin{theorem}[Theorem 2.18 \cite{DP1}]\label{homotopy}
There are isomorphisms of groups
\begin{align*}
\pi_i(\Aut(\OK)) &\cong \begin{cases} K_0(\cO_{\infty})^{\times}_+ &\text{ if } i=0; \\ K_i(\cO_{\infty}) &\text{ if } i \geq 1; \end{cases} \\
&= \begin{cases} \ZZ_2 &\text{ if } i=0; \\ \ZZ &\text{ if } i > 0 \text{ even}; \\ 0 &\text{ if } i \text{ odd}. \end{cases}
\end{align*}
\end{theorem}

To gain further insight into the automorphisms of $\cO_{\infty}$, in particular the outer automorphisms, we investigate an action of the infinite unitary group $U(\infty)$ on $\cO_{\infty}$. Note that this is one of the significant differences between the Cuntz algebra and the algebra of compact operators -- all automorphisms of $\cK$ are inner but this is not the case for $\cO_{\infty}$.

We follow the work of Enomoto, Fujii, Takehana and Watatani in describing an action of $U(\infty)$ on $\cO_{n}$ \cite{EFTW}. Let $M_n(\CC)$ be the algebra of all $n \times n$ matrices over $\CC$. For any $u = (u_{ij}) \in M_n(\CC)$, we define a map $\alpha_u: \cO_{\infty} \to \cO_{\infty}$ to act on the generators $S_1, S_2, \cdots$ of $\cO_{\infty}$ by
\[ \alpha_u(S_j) = \begin{cases} \sum\limits_{i=1}^{n} u_{ij} S_i &\text{ for } j = 1, \cdots, n; \\ S_j &\text{ for } j > n; \end{cases} \]
and extend this map to $\cO_{\infty}$ such that it is a homomorphism. It is clear that $\alpha_u \circ \alpha_{u'} = \alpha_{u u'}$, but for general $u$ this map does not give an automorphism of $\cO_{\infty}$. For example, taking $u$ to be the zero matrix we certainly do not obtain an automorphism, so we seek a class of matrices for which $\alpha_u$ is an automorphism.

\begin{lemma}[Lemma 4 \cite{EFTW}]
For $u = (u_{ij}) \in M_n(\CC)$, the map $\alpha_u$ defines an automorphism of $\cO_{\infty}$ if and only if $u \in U(n)$.
\end{lemma}

The proof of this lemma is a straightforward computation. This provides an action of $U(n)$ on $\cO_{\infty}$ for all $n = 1, 2, \cdots$ which we will see extends to an action of $U(\infty)$ on $\cO_{\infty}$. Here we are taking $U(\infty)$ to be the direct limit of $U(n) \xhookrightarrow{\iota} U(n+1)$ with $\iota(A) = \text{diag}(A,1)$. Taking $u \in U(\infty)$, there exists a finite representative $\hat{u} \in U(n)$ for some $n$, and every representative of $u$ will be of the form $\text{diag}(\hat{u},1,\cdots)$. Therefore all representatives of $u$ define the same action on $\cO_{\infty}$, so we define the action of $u$ on $\cO_{\infty}$ to be that of its finite representatives. This provides a map $\alpha: U(\infty) \to \Aut(\cO_{\infty})$ which satisfies some desirable properties; in particular, $\alpha$ lands in the group $\Out(\cO_{\infty})$ of outer automorphisms of $\cO_{\infty}$.

\begin{theorem}\label{action}
The map $\alpha: U(\infty) \to \Out(\cO_{\infty})$ is continuous and injective.
\end{theorem}

This is sketched for $\cO_n$ in Theorem 4 of \cite{EFTW}, and more details for continuity in the case of $\cO_{\infty}$ are provided in Theorem 2.2.7 of \cite{Brook}.

This action provides further insight into $\cO_{\infty}$ -- there are non-trivial outer automorphisms of this algebra, and a space as large as $U(\infty)$ can act effectively via these automorphisms. Given these preliminaries, we are now in a position to present the formulation of higher twisted $K$-theory by Pennig.

\subsection{Formulation}

The critical step in moving from the topological $K$-theory of a compact space $X$ to the twisted $K$-theory is replacing the algebra of continuous complex-valued functions on $X$ with the algebra of sections of a bundle of compact operators over $X$. The next natural extension of this would be to replace the algebra of compact operators with a different $C^*$-algebra, and one might expect that using a strongly self-absorbing $C^*$-algebra gives the desired construction. This is not quite the case; in fact, $\Aut(D)$ is contractible for all strongly self-absorbing $D$ as stated in Theorem \ref{contract} and therefore there are no non-trivial algebra bundles with fibre $D$ over $X$. Instead, we take the stabilisation of the strongly self-absorbing $C^*$-algebra, the automorphism group of which has a far more interesting homotopy type as mentioned in Theorem \ref{homotopy} for $\cO_{\infty}$ specifically. This culminates in one of the main theorems of Pennig and Dadarlat's paper.

\begin{theorem}[Theorem 3.8 (a), (b) \cite{DP1}]
Let $X$ be a compact metrisable space and let $D$ be a strongly self-absorbing $C^*$-algebra. The set $Bun_X(D \otimes \cK)$ of isomorphism classes of algebra bundles over $X$ with fibre $D \otimes \cK$ becomes an abelian group under the operation of tensor product. Furthermore, $B \Aut(D \otimes \cK)$ is the first space in a spectrum defining a cohomology theory ${E_D}^{\bullet}$.
\end{theorem}

While this result is interesting from a homotopy-theoretic point of view, it does not yet tell us that we can obtain the twists of $K$-theory using this construction. What we want is for the cohomology theory ${E_D}^{\bullet}$ to be $gl_1(KU)^{\bullet}$, so that the twists of $K$-theory may be identified with algebra bundles with fibre $D \otimes \cK$. The cohomology theory obtained, however, depends on the choice of strongly self-absorbing $C^*$-algebra. In fact, in the introduction of \cite{DP3} the authors claim that using the Jiang--Su algebra $\cZ$ yields a subset of twists of $K$-theory where $\ZZ_2 \times BU$ is replaced by $\{1\} \times BU$, and using a tensor product of a UHF algebra of infinite type with $\cO_{\infty}$ yields twists for localisations of $KU$. The full set of twists is the subject of the main theorem of \cite{DP3}.

\begin{theorem}[Adapted from Theorem 1.1 \cite{DP3}]\label{Bigthm}
The twists of $K$-theory over $X$ are classified by algebra bundles over $X$ with fibre $\OK$. To be more precise, ${E_{\cO_{\infty}}}^{\bullet} \simeq gl_1(KU)^{\bullet}$ and hence $Bun_X(\OK) \cong gl_1(KU)^1(X)$.
\end{theorem}

The significance of this theorem should not be overlooked. The proof requires heavy machinery from stable homotopy theory, much of which is built up over the series of three papers by the authors \cite{DP2,DP3,DP1}. It is only through this deep understanding of the abstract notion of twist that the authors were able to determine an appropriate model for the twists of $K$-theory using geometry and operator algebras.

Given this geometric notion of twist, we are now able to define the higher twisted $K$-theory groups. Note that our algebra bundles are fibre bundles where the trivialisation maps restrict to algebra isomorphisms on the fibres, and since we are working with $C^*$-algebras we see that the space of continuous sections of an algebra bundle over a compact Hausdorff space $X$ itself forms a $C^*$-algebra equipped with the induced norm and involution from the fibres. Furthermore, if $X$ is only a locally compact Hausdorff space then there is a sensible notion of a continuous section of an algebra bundle $\cA$ over $X$ vanishing at infinity, defined in much the same way as $C_0(X)$ and denoted $C_0(X,\cA)$. It is by taking the operator algebraic $K$-theory of this $C^*$-algebra that we wish to define higher twisted $K$-theory.

\begin{definition}
The higher twisted $K$-theory of the locally compact Hausdorff space $X$ with twist $\delta$ represented by the algebra bundle $\OK \to \cA_{\delta} \xrightarrow{\pi} X$ is defined to be $K^n(X,\delta) = K_n(C_0(X,\cA_{\delta}))$.
\end{definition}

This is actually not the definition originally given by Pennig -- he follows the homotopy-theoretic approach of using bundles of spectra and defines the higher twisted $K$-theory groups to be colimits of certain homotopy groups. The equivalent characterisation that we present above is given in his Theorem 2.7(c) \cite{Pennig}. As discussed earlier, replacing $\cO_{\infty}$ with other strongly self-absorbing $C^*$-algebras here will result in different versions of twisted $K$-theory which will likely also produce interesting results. For instance, Evans and Pennig use infinite UHF-algebras corresponding to twists of localisations of $K$-theory in a recent paper \cite{EvansP}.


As one might expect, there are also notions of relative higher twisted $K$-theory and higher twisted $K$-homology. The relative theory can be used to show that higher twisted $K$-theory forms a generalised cohomology theory; we will not need this, but the details are presented in Definition 2.6 and Theorem 2.7 of \cite{Pennig}. While we will be focusing on the twisted cohomology version of $K$-theory, the twisted homology version will be important in some computations and so we include this definition here.

\begin{definition}{\label{homology}}
The higher twisted $K$-homology of the locally compact Hausdorff space $X$ with twist $\delta$ represented by the algebra bundle $\OK \to \cA_{\delta} \xrightarrow{\pi} X$ is defined to be $K_n(X,\delta) = KK_n(C_0(X,\cA_{\delta}),\cO_{\infty})$.
\end{definition}

Again, this is not the way that Pennig initially defines higher twisted $K$-homology -- he introduces the topological definition via $\infty$-categories -- but the definition using $KK$-theory is shown to be equivalent in Corollary 3.5 using a Poincar\'e duality homomorphism \cite{Pennig}.

Of course, these definitions can be used to prove a variety of expected properties about higher twisted $K$-theory, including functoriality, cohomology properties and the existence of a graded module structure. These are proved in \cite{Pennig}, and the graded module structure in particular is explored in detail in Section 4.1 of \cite{Brook}. We note that the appropriate domain category for the higher twisted $K$-theory functor is the category of locally compact Hausdorff spaces equipped with algebra bundles with fibre $\OK$, such that a morphism $(X,\cA_{\delta_X}) \to (Y,\cA_{\delta_Y})$ is a proper map $f: X \to Y$ together with an algebra isomorphism $\theta: f^* \cA_{\delta_Y} \to \cA_{\delta_X}$, to ensure a relationship between the twist on $X$ and the twist on $Y$.

A corollary to Theorem \ref{Bigthm} is that taking the $K$-theory of a space $X$ twisted by the trivial algebra bundle simply returns the ordinary topological $K$-theory of $X$, while taking an algebra bundle which corresponds to a classical twist -- the tensor product of a bundle of compact operators with the trivial bundle with fibre $\cO_{\infty}$ -- returns the ordinary twisted $K$-theory of $X$. These results are reproved in Proposition 2.3.10 and Proposition 2.3.11 of \cite{Brook}. The former is a straightforward application of the K\"unneth theorem in $C^*$-algebraic $K$-theory, and the latter is proved using results on the sections of tensor product bundles. As expected, this shows that higher twisted $K$-theory contains all of the information of topological and classical twisted $K$-theory, along with a great deal more.

Another basic property of higher twisted $K$-theory which we will make use of is the Mayer--Vietoris sequence, which will be of critical importance for computations.

\begin{proposition}[Theorem 2.7(f) \cite{Pennig}]\label{MV}
Let $X = U_1 \cup U_2$ for closed sets $U_k$ whose interiors cover $X$. Let $i_k: U_k \to X$ and $j_k: U_1 \cap U_2 \to U_k$ denote inclusion, and $\delta|_{U_k} = i_k^* \delta$ and $\delta|_{U_1 \cap U_2} = (i_1 \circ j_1)^* \delta$ denote restriction of the twist $\delta$ over $X$ to the corresponding subspaces. Then there is a six-term Mayer--Vietoris sequence as follows:
\[
\begin{tikzcd}
K^0(X,\delta) \ar{r}{(i_1^*,i_2^*)} & K^0(U_1,\delta|_{U_1})\oplus K^0(U_2,\delta|_{U_2}) \ar{r}{j_1^*-j_2^*} & K^0(U_1 \cap U_2,\delta|_{U_1 \cap U_2}) \ar{d}{\partial_0} \\
K^1(U_1 \cap U_2,\delta|_{U_1 \cap U_2}) \ar{u}{\partial_1} & K^1(U_1,\delta|_{U_1}) \oplus K^1(U_2,\delta|_{U_2}) \ar{l}{j_1^*-j_2^*} & K^1(X,\delta) \ar{l}{(i_1^*,i_2^*)}.
\end{tikzcd}
\]
\end{proposition}

The following is an original proposition from \cite{Brook} relating higher twisted $K$-theory and higher twisted $K$-homology, which we will later use in computations.

\begin{proposition}{\label{K-homology}}
Assume that the algebra of sections vanishing at infinity of any algebra bundle with fibres isomorphic to $\OK$ over a locally compact space $X$ is contained in the bootstrap category of $C^*$-algebras defined in Definition 22.3.4 of \cite{BlackK}. If the higher twisted $K$-theory of $X$ is a direct sum of finite torsion groups, then the higher twisted $K$-theory and higher twisted $K$-homology of $X$ are isomorphic with a degree shift.
\end{proposition}

\begin{proof}
The higher twisted $K$-theory and $K$-homology groups can be related by the universal coefficient theorem in $KK$-theory as in Theorem 23.1.1 of \cite{BlackK}, which states that
\[ 0 \to \Ext^1_{\ZZ}(K_*(A),K_*(B)) \to KK^*(A,B) \to \Hom_{\ZZ}(K_*(A),K_*(B)) \to 0 \]
is a short exact sequence whose first map has degree 0 and second map has degree 1 with respect to the grading, if $A$ and $B$ are seperable and $A$ is in the bootstrap category of $C^*$-algebras defined in Definition 22.3.4 of \cite{BlackK}. In order to obtain higher twisted $K$-homology as the $KK$-group in this sequence, we let $A$ be the space of sections of the algebra bundle representing the twist and let $B = \cO_{\infty}$. Then assuming that $A$ is in the bootstrap category, the short exact sequence becomes
\begin{align*}
0 \to \Ext^1_{\ZZ}(K^{n+1}(X,\delta),K_0(&\cO_{\infty}))\oplus \Ext^1_{\ZZ}(K^n(X,\delta),K_1(\cO_{\infty}) \to K_n(X,\delta) \\
\to \Hom_{\ZZ}(&K^{n}(X,\delta),K_0(\cO_{\infty}) \oplus \Hom_{\ZZ}(K^{n+1}(X,\delta), K_1(\cO_{\infty}) \to 0
\end{align*}
for each $n$. Using the $K$-theory of $\cO_{\infty}$, this reduces to
\[ 0 \to \Ext^1_{\ZZ}(K^{n+1}(X,\delta),\ZZ) \to K_n(X,\delta) \to \Hom_{\ZZ}(K^n(X,\delta),\ZZ) \to 0. \]
Now, if $K^{n}(X,\delta) \cong \bigoplus_{k} \ZZ_{m_{k,n}} $
for some finite sequence $\{m_{k,n}\}$ and $n=0, 1$ as per our assumption, then the $\Hom$ group will be trivial and the $\Ext$ group will become
\begin{align*}
\Ext^1_{\ZZ}\left(\bigoplus\limits_{k} \ZZ_{m_{k,n}},\ZZ\right) &= \bigoplus\limits_{k} \Ext^1_{\ZZ}(\ZZ_{m_{k,n}},\ZZ) \\
&= \bigoplus\limits_{k} \ZZ_{m_{k,n}} \\
&= K^n(X,\delta),
\end{align*}
since the $\Ext$ functor is additive in the first variable and $\Ext^1_{\ZZ}(\ZZ_m,G) \cong G / mG$ by properties in Section 3.1 of \cite{HatcherAT}. Therefore the short exact sequence provides an isomorphism $K^{n+1}(X,\delta) \cong K_n(X,\delta)$ as required.
\end{proof}

\begin{remark}\label{bootstrap}
We highlight the assumption made in the statement of Proposition \ref{K-homology}; that the algebra of sections vanishing at infinity of any algebra bundle with fibres isomorphic to $\OK$ over a locally compact space $X$ is contained in the bootstrap category of $C^*$-algebras defined in Definition 22.3.4 of \cite{BlackK}, and thus satisfies the universal coefficient theorem. This assumption is valid when $\OK$ is replaced by $\cK$, and it is true for both $C_0(X)$ for any locally compact Hausdorff space $X$ as well as for $\OK$ itself. In fact, it is conjectured that every seperable nuclear $C^*$-algebra satisfies the universal coefficient theorem in $KK$-theory. We note that this Proposition is only used in one computation; Theorem \ref{concerned} in the final section.
\end{remark}

\subsection{Links to cohomology}

In the classical case, twists of $K$-theory were not only viewed as algebra bundles; often this viewpoint was complemented using cohomology classes. Classical twists were often viewed as elements of $H^3(X,\ZZ)$, and it is precisely the Dixmier--Douady theory which provided a link between these cohomology classes and algebra bundles with fibre $\cK$. This raises the question as to whether there is any link between algebra bundles over a space $X$ with fibre $\OK$ and the cohomology of $X$, and the higher Dixmier--Douady theory posed by Pennig and Dadarlat in relation to strongly self-absorbing $C^*$-algebras is the key to understanding this. The following results are discussed in generality for all strongly self-absorbing $C^*$-algebras in Section 4 of \cite{DP1}, but we will specifically consider the use of $\cO_{\infty}$.

What we desire is a way to interpret the twists of $K$-theory, i.e.\ the elements of the first group of some generalised cohomology theory $E^1_{\cO_{\infty}}(X)$, in terms of the ordinary cohomology of $X$. This is precisely what a spectral sequence allows. As with any generalised cohomology theory, there is an Atiyah--Hirzebruch spectral sequence converging to ${E_{\cO_{\infty}}}^{\bullet}$, the coefficients of which are given in Theorem \ref{homotopy}. The $E_2$-term of this spectral sequence is as follows.

\begin{center}
\begin{tikzpicture}
  \matrix (m) [matrix of math nodes,
    nodes in empty cells,nodes={minimum width=2.3ex,
    minimum height=5ex,outer sep=-5pt},
    column sep=1ex,row sep=1ex]{
    				& &  0  &  1  &  2 & 3 & \\
		0    	& &	H^0(X,\ZZ_2)  & H^1(X,\ZZ_2) & H^2(X,\ZZ_2) & H^3(X,\ZZ_2) & \\
		-1     	& &	0 &  0  & 0 & 0 & \\
		-2    	& &	H^0(X,\ZZ)  & H^1(X,\ZZ) & H^2(X,\ZZ) & H^3(X,\ZZ)  &  \\
		-3    	& &	0 &  0  & 0 & 0 & \\
		-4    	& &	H^0(X,\ZZ)  & H^1(X,\ZZ) & H^2(X,\ZZ) & H^3(X,\ZZ)  &  \\};
\draw[thick] (m-1-1.south) -- (m-1-7.south) ;
\draw[thick] (m-1-2.north) -- (m-6-2.south) ;
\end{tikzpicture}
\end{center}

Note that we use different indexing conventions to distinguish between different Atiyah--Hirzebruch spectral sequences. For the spectral sequence for ${E_{\cO_{\infty}}}^{\bullet}$ we adopt the indexing used by Pennig and Dadarlat where negative indices are used, while for the spectral sequence for higher twisted $K$-theory developed in Section 4.1 we use standard indexing.

At this stage there are complications. The differentials in this sequence are unknown, and even if they were known there may be non-trivial extension problems in determining $E^1_{\cO_{\infty}}(X)$. At this point, Pennig and Dadarlat restrict to the setting in which $X$ has torsion-free cohomology, as this implies that the differentials of the sequence are necessarily zero as they are torsion operators, meaning that their image is torsion, as shown in Theorem 2.7 of \cite{Arlettaz}. It is then clear that there will be no extension problems, as there are no non-trivial extensions of free abelian groups and the only torsion will be in the final summand $H^1(X,\ZZ_2)$. Thus we obtain the following, noting that to apply the spectral sequence we must be working with a finite connected CW complex.

\begin{theorem}[Corollary 4.7(ii) \cite{DP1}]{\label{cohomology}}
Let $X$ be a finite connected CW complex such that the cohomology ring of $X$ is torsion-free. Then
\[ E^1_{\cO_{\infty}}(X) \cong Bun_X(\OK) \cong H^1(X,\ZZ_2) \oplus \bigoplus\limits_{k \geq 1} H^{2k+1}(X,\ZZ). \]
\end{theorem}

This shows that there is a relationship between the twists of $K$-theory and cohomology, at least in the restrictive case when $X$ has torsion-free cohomology. Even when the cohomology of $X$ has torsion, the spectral sequence argument shows that the twists will correspond to some subset of these odd-degree cohomology groups depending on differentials and extension problems. This also confirms that, in this case, the classical twists classified by $H^1(X,\ZZ_2) \oplus H^3(X,\ZZ)$ are indeed twists of $K$-theory from the perspective of homotopy theory, and provides insight into Pennig's chosen name -- ``higher'' twisted $K$-theory. The twists introduced by Pennig are higher in the sense that they can be represented by higher degree cohomology classes, as opposed to the classical degree 1 and 3 twists.

Note that henceforth, when we are in a setting in which Theorem \ref{cohomology} applies, we will identify the twists of $K$-theory over $X$ with the odd-degree integral cohomology classes of $X$. Given a twist we may view this as a cohomology class, and given a cohomology class this will represent a twist. This will be particularly important in the development of spectral sequences and in computations.

There are some slightly more general statements that can be made even if the cohomology of the base space has torsion. This will be the case if the torsion does not have any effect on the previous argument, i.e.\ if the relevant differentials are necessarily zero and there are no extension problems. Since we are only interested in the degree 1 group of this cohomology theory, only the groups whose row and column index sum to 1 are relevant, and so we only need to worry about the differentials entering and leaving these groups. If, for instance, only the odd cohomology groups of $X$ were torsion-free, the differentials between these relevant groups would all necessarily be zero, and we would be able to reach the same conclusion that twists correspond to odd-degree cohomology classes. Even these slightly relaxed assumptions allow for a wider class of spaces to be considered, including real projective spaces and lens spaces.

\subsection{Reformulation using Fredholm operators}

To conclude this section, we will present an alternative definition of the higher twisted $K$-theory groups which will prove useful in explicitly describing elements of these groups. We follow an approach of Rosenberg presented in \cite{RosenbergTwist} about classical twisted $K$-theory.

We will need to slightly shift our perspective from that of algebra bundles to that of principal bundles. Recall that there is a correspondence between algebra bundles over $X$ with fibre $\OK$ and principal $\Aut(\OK)$-bundles over $X$ displayed in the following diagram, since $\Aut(\OK)$ acts effectively on $\OK$ by automorphisms.
\[
\begin{tikzcd}
\Aut(\OK) \arrow[r] & \cE_{\delta} \arrow[d] \arrow[rr,bend left,dashed,"\text{Associated bundle}"] & &\cE_{\delta} \times_{\Aut(\OK)} (\OK) \arrow[d]\\
& X & & X \ar[ll, bend left,dashed,"\text{Transition functions}"]
\end{tikzcd}
\]
This is a bijective correspondence; moving from one perspective to the other and back again yields an isomorphic bundle, and therefore we may view either of the objects above as a twist of $K$-theory over $X$. Viewing a twist $\delta$ as a principal $\Aut(\OK)$-bundle as opposed to an algebra bundle, we define the $K$-theory of $X$ twisted by $\delta$ to be the operator algebraic $K$-theory of the continuous sections vanishing at infinity of the associated algebra bundle to agree with our previous definition.

We also need the notion of Fredholm operators on Hilbert $C^*$-modules to present this result. The theory of Hilbert modules and the operators which act on them is developed in Part III of \cite{WO}, including their relevance in $C^*$-algebraic $K$-theory. Key results relevant to higher twisted $K$-theory are summarised in Section 1.2.3 of \cite{Brook}. We will make use of the standard Hilbert $(\OK)$-module and the space of Fredholm operators on this module, denoted $\cH_{\OK}$ and $\Fred_{\OK}$ respectively. We are now able to state the main result of this section.

\begin{theorem}{\label{topological}}
Let $X$ be a compact Hausdorff space and $\cE_{\delta}$ a principal $\Aut(\OK)$-bundle over $X$ representing a twist $\delta$. There are natural identifications
\[ K^0(X,\delta) \cong [\cE_{\delta}, \Fred_{\OK}]^{\Aut(\OK)} \]
and
\[ K^1(X,\delta) \cong [\cE_{\delta}, \Omega \Fred_{\OK}]^{\Aut(\OK)} \]
where $[ \enspace, \enspace]^{\Aut(\OK)}$ denotes the unbased homotopy classes of $\Aut(\OK)$-equivariant maps, i.e.\ $\pi_0(C(\enspace,\enspace)^{\Aut(\OK)})$, and $\Omega \Fred_{\OK}$ is the based loop space of $\Fred_{\OK}$, i.e. the space of continuous maps $S^1 \to \Fred_{\OK}$ with $1 \in S^1$ mapped to $I \in \Fred_{\OK}$.
\end{theorem}

Note that $\Aut(\OK)$ acts on $\cE_{\delta}$ as the structure group of the principal bundle and acts on $\Fred_{\OK}$ via conjugation in the same way that $PU$ acts on $\Fred$, meaning that 
\begin{equation}\label{actionfred}
F \cdot T = (T^{-1})_{\cH} F T_{\cH}
\end{equation}
where $T \in \Aut(\OK)$, $F \in \Fred_{\OK}$ and we denote the map induced on the standard Hilbert $\OK$-module $\cH_{\OK}$ by applying $T$ pointwise by $T_{\cH}$. The action on $\Omega \Fred_{\OK}$ is defined to be this action at every point in the loop.

Before proceeding with the proof, we need a standard lemma about sections of principal bundles.

\begin{lemma}{\label{equiv}}
Let $\cE$ be a principal $G$-bundle over a topological space $X$ and suppose that $G$ has a continuous and effective left action on a topological space $F$. Then the space of sections of the associated fibre bundle $\cE \times_G F$ over $X$ can be identified with the space of $G$-equivariant maps $\cE \to F$.
\end{lemma}

This is part of the proof of Proposition 1.3 in \cite{RosenbergTwist}. We are now equipped to prove Theorem \ref{topological}.

\begin{proof}[Proof of Theorem \ref{topological}]
Recall that $K_0(C_0(X,\cA_{\delta}))$ can be identified with the group of path components of the Fredholm operators on the standard Hilbert $C_0(X,\cA_{\delta})$-module.
This means that
\begin{align*}
K^0(X,\delta) &= K_0(C(X,\cE_{\delta} \times_{\Aut(\OK)} (\OK))) \\
&= \pi_0(\Fred_{C(X,\cE_{\delta} \times_{\Aut(\OK)} (\OK))}),
\end{align*}
and applying Lemma \ref{equiv} allows us to replace the continuous sections of the algebra bundle associated to $\cE_{\delta}$ with the $\Aut(\OK)$-equivariant maps from $\cE_{\delta}$ to $\OK$.
This allows us to conclude that
\[ \Fred_{C(X,\cE_{\delta} \times_{\Aut(\OK)} (\OK))} = C(\cE_{\delta},\Fred_{\OK})^{\Aut(\OK)}, \]
and so
\begin{align*}
K^0(X,\delta) = \pi_0(C(\cE_{\delta},\Fred_{\OK})^{\Aut(\OK)}) = [\cE_{\delta},\Fred_{\OK}]^{\Aut(\OK)}
\end{align*}
as required. In order to obtain the result for $K^1$, we recall that $K_1(A) = K_0(SA)$ for a $C^*$-algebra $A$ where $SA$ denotes the suspension. In this case, we are interested in the $C^*$-algebra $SC(X,\cE_{\delta} \times_{\Aut(\OK)} (\OK))$, which can be viewed as
\[ \{f: S^1 \to C(X,\cE_{\delta} \times_{\Aut(\OK)} (\OK)) \mid \text{continuous}, f(1) = 0 \}. \]
We will suppress the continuity of the function and the fact that $f(1) = 0$ for brevity, but the same conditions are required to hold in the following sets where $0$ is taken to be the additive identity in each case. In the same way as above, we can view the Fredholm operators on the standard Hilbert $C^*$-module of this $C^*$-algebra as
\begin{align*}
&\{f: S^1 \to C(X,\cE_{\delta} \times_{\Aut(\OK)} \Fred_{\OK}) \} \\
= \enspace & \{f: S^1 \to C(\cE_{\delta}, \Fred_{\OK})^{\Aut(\OK)} \} \\
= \enspace & C(\cE_{\delta} \times S^1, \Fred_{\OK})^{\Aut(\OK)} \\
= \enspace & C(\cE_{\delta}, C(S^1, \Fred_{\OK}))^{\Aut(\OK)} \\
= \enspace & C(\cE_{\delta}, \Omega \Fred_{\OK})^{\Aut(\OK)}.
\end{align*}
Thus we may conclude that
\[ K^1(X,\delta) = \pi_0(C(\cE_{\delta}, \Omega \Fred_{\OK})^{\Aut(\OK)}) = [\cE_{\delta}, \Omega \Fred_{\OK}]^{\Aut(\OK)} \]
as required.
\end{proof}

In the case that principal $\Aut(\OK)$-bundles over a space $X$ can be explicitly described, this provides a useful way of expressing elements in the higher twisted $K$-theory groups of $X$. This formulation can also be extended to obtain expressions for the higher twisted $K$-theory groups of higher degree, where we see that
\[ K^n(X,\delta) = [\cE_{\delta}, \Omega^n \Fred_{\OK}]^{\Aut(\OK)}. \]
Although this may be a useful expression, the statement of the theorem already covers the important cases since Bott periodicity implies that everything will reduce to these two groups.

\section{Geometric constructions of higher twists}

Whilst knowing that the twists of $K$-theory over a space $X$ may be identified with algebra bundles over $X$ with fibres isomorphic to $\OK$ is useful in its own right, the method of proof does not provide us with an explicit construction of a bundle representing a homotopy-theoretic twist. In particular, since the definition of higher twisted $K$-theory involves the algebra of sections of such an algebra bundle, it is easier to compute higher twisted $K$-theory groups when there is an explicit bundle to work with. In the general case, even classifying the $\OK$ bundles over $X$ is a difficult task. In the case that twists can be identified with cohomology classes, however, associating an explicit bundle to each cohomology class will allow for simpler methods of computation. We will explore this in two special cases; when the base space is an odd-dimensional sphere and when the twisting class is decomposable.


\subsection{The clutching construction}

We begin by considering spheres; the clutching construction is the key to constructing these algebra bundles in this simple case. Recall that the clutching construction in general takes a fibre bundle over each hemisphere and glues them together using a gluing function from the equatorial sphere into the structure group of the fibre bundle. More general versions of the construction exist for other spaces which we will mention briefly, but it is particularly useful for the spheres because the gluing map can be viewed as an element of a homotopy group of the structure group, which provides a way of classifying these maps in cases that the homotopy type of the structure group is understood. It is also useful because the hemispheres are contractible, meaning that without loss of generality all fibre bundles over the hemispheres can be assumed trivialised. In particular, a principal $\Aut(\OK)$-bundle over $S^n$ may be constructed by specifying a map $f: S^{n-1} \to \Aut(\OK)$ which will glue trivial bundles over the upper and lower hemispheres $D^n_+$ and $D^n_-$ respectively. More precisely, we make the following definition.

\begin{definition}{\label{clutch}}
Let $f: S^{n-1} \to \Aut(\OK)$ be a continuous map. The \textit{clutching bundle} $\cE_f$ over $S^n$ associated to $f$ is defined to be the quotient of the disjoint union \[(D_+^n \times \Aut(\OK)) \amalg (D_-^n \times \Aut(\OK))\] under $(x,T) \sim (x,f(x) \circ T)$ for all $x \in S^{n-1}$ and $T \in \Aut(\OK)$.
\end{definition}

Note that technically this equivalence is between points $(x_+,T)$ and $(x_-, f(x) \circ T)$ where $x_+ \in D^{n}_+$ and $x_- \in D^{n}_-$ both represent the same point $x \in S^{n-1}$, but we will suppress these subscripts. We could equivalently have constructed an algebra bundle with fibres isomorphic to $\OK$ over $S^n$ by replacing $T \in \Aut(\OK)$ with $o \in \OK$, but the principal bundle construction will be more convenient due to Theorem \ref{topological} being stated in terms of principal bundles. The following lemma shows that this distinction is unimportant.

\begin{lemma}
The algebra bundle associated to a principal bundle constructed via the clutching construction is isomorphic to the algebra bundle constructed directly via the clutching construction with the same gluing map.
\end{lemma}

The proof follows easily from definitions; see Lemma 3.1.4 of \cite{Brook}. An added benefit of the clutching construction is that there is a simple way to describe the sections of a clutching bundle using sections of the trivial bundles over the hemispheres. In particular, a section of a clutching bundle can be identified with sections of the trivial bundles which interact via the gluing map as follows. Note that we abbreviate $C(D^n_{+},\Aut(\OK)) \oplus C(D^n_-, \Aut(\OK))$ as $C(D^n_+ \amalg D^n_-,\Aut(\OK))$ for brevity.

\begin{lemma}\label{sections}
The space of sections of the clutching bundle over $S^n$ associated to the gluing map $f: S^{n-1} \to \Aut(\OK)$ is of the form
\[ C(S^n,\cE_f) = \{ (g,h) \in C(D^n_+ \amalg D^n_-,\Aut(\OK)) : g(x) = f(x) \cdot h(x) \text{ for all } x \in S^{n-1} \}. \]
\end{lemma}

The proof follows using the definition of the clutching bundle; see Lemma 3.1.2 of \cite{Brook} for more details. This result is particularly useful because the higher twisted $K$-theory groups are defined using the algebra of sections, and so having an explicit realisation of this algebra will allow computations to be performed more easily.

Now, we have an explicit construction of a bundle from a gluing map, but we want to be able to explicitly construct a bundle from a cohomology class of the sphere. To move towards this goal, we recall that any principal $\Aut(\OK)$-bundle over $S^n$ can be constructed in this way, and furthermore that the isomorphism class of the bundle depends only on the homotopy class of the gluing map.

\begin{proposition}{\label{correspondence}}
There is a bijective correspondence between the set of isomorphism classes of principal $\Aut(\OK)$-bundles over $S^n$ and $\pi_{n-1}(\Aut(\OK))$.
\end{proposition}

The proof for vector bundles is given in Theorem 2.7 of \cite{fiber}, which is easily adapted to this case in Proposition 3.1.3 of \cite{Brook}.

The map defined by the bijective correspondence in Proposition \ref{correspondence} is an explicit realisation of the isomorphism induced by viewing $S^n$ as the suspension $\Sigma S^{n-1}$:
\begin{align*}
[S^n, B\Aut(\OK)] &= [\Sigma S^{n-1}, B\Aut(\OK)] \\
&= [S^{n-1}, \Omega B \Aut(\OK)] \\
&\cong [S^{n-1}, \Aut(\OK)].
\end{align*}

Now, since the cohomology groups of the spheres are torsion-free, we see via Theorem \ref{cohomology} that the twists of $K$-theory over the spheres are classified by their odd-degree cohomology groups, i.e.\ $H^{2n+1}(S^{2n+1},\ZZ) \cong \ZZ$ for $n\geq 1$. Finally, these correspondences
\[ H^{2n+1}(S^{2n+1},\ZZ) \cong [S^{2n+1},B \Aut(\OK)] \cong \pi_{2n}(\Aut(\OK)) \]
allow us to obtain explicit geometric representatives for twists over $S^{2n+1}$ given in terms of cohomology classes. Letting $[\delta_0] \in H^{2n+1}(S^{2n+1},\ZZ) \cong \ZZ$ denote a generator and taking any $N \in \ZZ$, we see that the bundle representing the twist $N[\delta_0]$ is constructed via a degree $N$ gluing map, i.e.\ $N$ times the generator of $\pi_{2n}(\Aut(\OK))$ corresponding to $[\delta_0]$ under the above identification. Using this result, we are able to explicitly compute the higher twisted $K$-theory of the odd-dimensional spheres directly from the definition rather than by using higher-powered machinery such as spectral sequences.

We also note that the construction can only produce trivial principal $\Aut(\OK)$-bundles over even-dimensional spheres, which is expected from Theorem \ref{cohomology}. This is because a bundle over $S^{2n}$ would come from a gluing map $S^{2n-1} \to \Aut(\OK)$, but the homotopy group $\pi_{2n-1}(\Aut(\OK))$ is trivial as stated in Theorem \ref{homotopy} and so every gluing map is homotopic to a constant. This means that any gluing map will construct a bundle which is isomorphic to the trivial bundle, which agrees with the fact that the odd-dimensional cohomology of $S^{2n}$ is trivial.

We will briefly make some more general remarks regarding the clutching construction. Given any cover of a space $X$ and a principal bundle over the disjoint union of the cover with certain isomorphism conditions imposed on points in the disjoint union which are identified to construct $X$, a principal bundle over $X$ can be constructed. For simplicity, we restrict our attention to spaces which can be covered by two sets as was the case with the sphere. For example, by viewing complex projective space $\CC P^n$ as the quotient of the disk $D^{2n}_+$ by the equivalence relation identifying antipodal points on the boundary, $\CC P^n$ can be covered by the image of a set containing a neighbourhood of the boundary of the disk under the projection map and a set which does not contain the boundary. The latter of these sets is contractible but the former is topologically more complicated, meaning that the principal bundles over this set would need to be better understood in order to construct any general principal bundle over $\CC P^n$. This method could still be used to construct some principal $\Aut(\OK)$-bundles over $\CC P^n$, even if it is not possible to construct all in this way. More simply, this construction can be used to construct principal $\Aut(\OK)$-bundles over products of spheres containing at least one odd-dimensional sphere where said sphere is split into two hemispheres as above, as we will observe in Section 6.2.

\subsection{Decomposable cohomology classes}

Another approach to constructing geometric representatives for twists of $K$-theory is to consider general spaces $X$ but to simplify the twisting cohomology class. One way to do this is to consider decomposable classes, because there already exist geometric representatives for some low-dimensional cohomology classes.

To begin with, let $X$ be a finite connected CW complex with torsion-free cohomology so that we are in the setting of Theorem \ref{cohomology}. Suppose that $\delta \in H^5(X,\ZZ)$ decomposes as the cup product $\delta = \alpha \cup \beta$ with $\alpha \in H^2(X,\ZZ)$ and $\beta \in H^3(X,\ZZ)$. By the standard identification $H^n(X,\ZZ) \cong [X, K(\ZZ,n)]$ where $K(G,n)$ denotes an Eilenberg--Mac Lane space, i.e.\ a space whose only non-trivial homotopy group is $G$ in degree $n$, and the fact that there exist simple geometric models for $K(\ZZ,n)$ in the case that $n=1, 2, 3$, we identify $\delta$ with $H_5: X \to K(\ZZ,5)$ such that $H_5 = H_2 \wedge H_3$ with $H_2: X \to BU(1)$ and $H_3: X \to BPU$. Then $H_2$ determines a principal $U(1)$-bundle
\[
\begin{tikzcd}
U(1) \arrow[r] &P_{H_2} \arrow[d,"\pi_{H_2}"] \\
&X
\end{tikzcd}
\]
with Chern class $\alpha$, and similarly $H_3$ determines a principal $PU$-bundle
\[
\begin{tikzcd}
PU \arrow[r] &Q_{H_3} \arrow[d,"\pi_{H_3}"] \\
&X
\end{tikzcd}
\]
with Dixmier--Douady invariant $\beta$. We form the fibred product bundle over $X$, whose total space is $P_{H_2} \times_X Q_{H_3} = \{(p,q) \in P_{H_2} \times Q_{H_3} : \pi_{H_2}(p) = \pi_{H_3}(q) \},$ and this gives us a principal $U(1) \times PU$-bundle
\[
\begin{tikzcd}
U(1) \times PU \arrow[r] &P_{H_2} \times_X Q_{H_3} \arrow[d,"\pi"] \\
&X
\end{tikzcd}
\]
of which $\alpha \cup \beta$ will be an invariant. This is because the fibred product of principal bundles is the pullback of the direct product of principal bundles under the diagonal map, and the cup product in cohomology is the pullback of the external product under the diagonal map. The direct product of principal bundles corresponds to the external product of cohomology classes, so the fibred product of principal bundles corresponds to the cup product in cohomology.

Now, to this principal bundle we wish to associate a principal $\Aut(\OK)$-bundle over $X$, to obtain a twist of $K$-theory. This is done by defining an injective group homomorphism from the structure group $U(1) \times PU$ into the automorphism group $\Aut(\OK)$, or equivalently an effective action of the structure group $U(1) \times PU$ on the algebra $\OK$. Since $PU$ is isomorphic to the automorphism group of $\cK$ by conjugation, we have the obvious action $PU \xrightarrow{\cong} \Aut(\cK)$. We seek an effective action of $U(1)$ on $\cO_{\infty}$.


As noted in Section 3 of \cite{KishimotoCross}, there is a one-parameter automorphism group of $\cO_{\infty}$ obtained by scaling the generators as follows. Letting $\lambda_k$ for $k = 1, 2, \cdots$ be a sequence of real constants, we obtain a map $\gamma: \RR \to \Aut(\cO_{\infty})$ defined by $\gamma_t(S_k) = e^{i \lambda_k t} S_k$ for $k = 1, 2, \cdots$ where the $S_k$ are the generators in the definition of the Cuntz algebra $\cO_{\infty}$. Then taking $\lambda_k = 2k\pi$ we see that $\gamma$ is periodic in $t$ with a period of 1. In fact, this is a special case of the action that we described in Theorem \ref{action}, and thus we can view it as a map $\gamma: U(1) \to \Out(\cO_{\infty})$, yielding the desired action.

Out of our maps $U(1) \to \Aut(\cO_{\infty})$ and $PU \to \Aut(\cK)$, we obtain the product map $U(1) \times PU \to \Aut(\cO_{\infty}) \times \Aut(\cK)$. Then by Corollary T.5.19 of \cite{WO} we see that the tensor product of two automorphisms of $C^*$-algebras is an automorphism of the tensor product algebra, so $\Aut(\cO_{\infty}) \times \Aut(\cK) \subset \Aut(\OK)$. Finally, since the map $U(1) \to \Aut(\cO_{\infty})$ is given by scaling generators whereas $PU = U(\cH) / U(1)$ acts by conjugation on $\cK$, there is no non-trivial action of the $U(1)$ factor on the $\cK$ component or of the $PU$ factor on the $\cO_{\infty}$ component and hence the map $U(1) \times PU \to \Aut(\OK)$ that we have constructed is injective.

Thus we may form the associated bundle $(P_{H_2} \times_X Q_{H_3}) \times_{U(1) \times PU} \Aut(\OK)$, which is a principal $\Aut(\OK)$-bundle over $X$. As the 5-class $\alpha \cup \beta$ is an invariant of the principal $U(1) \times PU$-bundle, this is a prime candidate for the principal $\Aut(\OK)$-bundle over $X$ which corresponds to the class $\alpha \cup \beta$ under the isomorphism of Dadarlat and Pennig. Due to the inexplicit nature of the isomorphism, however, it is not immediate that this will indeed be the correct bundle. In order to get around this issue, we note that Pennig and Dadarlat use an Atiyah--Hirzebruch spectral sequence to determine ${E_{\cO_{\infty}}}^1(X) = [X,B\Aut(\OK)]$ in terms of cohomology. Following their argument, the same spectral sequence yields the standard isomorphism $[X,K(\ZZ,5)] \cong H^5(X,\ZZ)$. This allows us to conclude that the diagram
\begin{equation}\label{diagram}
\begin{tikzcd}
\left[X, K(\ZZ,5)\right] \ar[hookrightarrow]{d} \ar{r}{\cong} & H^5(X,\ZZ) \ar[hookrightarrow]{d} \\
\left[X,B\Aut(\OK)\right] \ar{r}{\cong} &H^1(X,\ZZ_2) \oplus \bigoplus\limits_{k \geq 1} H^{2k+1}(X,\ZZ)
\end{tikzcd}
\end{equation}
commutes, where we are viewing the elements of $[X,K(\ZZ,5)]$ as principal $U(1) \times PU$-bundles over $X$ obtained using the fibred product construction so the left vertical map takes such a bundle to the associated $\Aut(\OK)$-bundle via our injective group homomorphism.
This allows us to conclude that the principal bundle $(P_{H_2} \times_X Q_{H_3}) \times_{U(1) \times PU} \Aut(\OK)$ truly does correspond to the 5-class $\alpha \cup \beta$ under the isomorphism of Dadarlat and Pennig.

We now extend this argument to the case in which $\delta$ is a general element of the cup product of $H^2(X,\ZZ)$ and $H^3(X,\ZZ)$, i.e.\ $\delta$ is given by a sum of $N$ decomposable classes of the form considered above. We take $\alpha \in H^2(X,\ZZ^N)$ and $\beta \in H^3(X,\ZZ^N)$ such that $\delta = \inner{\alpha}{\beta}$ where $\inner{\cdot}{\cdot}$ is the pairing $H^2(X,\ZZ^N) \times H^3(X,\ZZ^N) \to H^5(X,\ZZ)$ given by the cup product and the standard inner product $\ZZ^N \times \ZZ^N \to \ZZ$. As above, we identify $\alpha$ with a map $H_2: X \to BU(1)^N$ and $\beta$ with a map $H_3: X \to BPU^N$, and form the principal torus bundle
\[
\begin{tikzcd}
U(1)^N \arrow[r] &P_{H_2} \arrow[d,"\pi_{H_2}"] \\
&X
\end{tikzcd}
\]
with Chern class $\alpha$ and the principal $PU^N$-bundle
\[
\begin{tikzcd}
PU^N \arrow[r] &Q_{H_3} \arrow[d,"\pi_{H_3}"] \\
&X
\end{tikzcd}
\]
with Dixmier--Douady invariant $\beta$. Once again we take the fibred product to obtain the principal $U(1)^N \times PU^N$-bundle $P_{H_2} \times_X Q_{H_3}$ over $X$
with invariant $\delta$. We now require an injective map $U(1)^N \times PU^N \to \Aut(\OK)$ to adapt the previous argument and construct the associated bundle.

Again applying Corollary T.5.19 of \cite{WO} and using the fact that $\cO_{\infty}$ is nuclear, we see that $\Aut(\cO_{\infty})^N \subset \Aut(\cO_{\infty}^{\otimes N}) = \Aut(\cO_{\infty})$. We may then combine two copies of the action $\gamma, \gamma': U(1) \to \Aut(\cO_{\infty})$ described previously to form an injective map $U(1)^2 \to \Aut(\cO_{\infty}^{\otimes 2})$ given by $(\gamma \otimes \gamma')_{(t,t')}(S_k \otimes S_{k'}) = e^{2\pi i(kt + k't')} S_k \otimes S_{k'}$. This argument can be extended to $N$ copies of the action in the same way, giving an injective map $U(1)^N \to \Aut(\cO_{\infty})$. More simply, the map $PU^N \to \Aut(\cK)$ is injective because the automorphisms do not involve scaling by constants. Thus we obtain our desired injective group homomorphism $U(1)^N \times PU^N \to \Aut(\OK)$.

We then construct the associated bundle $(P_{H_2} \times_X Q_{H_3}) \times_{U(1)^N \times PU^N} \Aut(\OK)$ over $X$, at which point the commutative diagram (\ref{diagram}) again allows us to conclude that this principal bundle does correspond to $\delta$ under the isomorphism. We have proved the following.

\begin{theorem}\label{decomposable_theorem}
Let $X$ be a finite connected CW complex with torsion-free cohomology, and take $\alpha \in H^2(X,\ZZ^N)$ and $\beta \in H^3(X,\ZZ^N)$ with $\delta = \inner{\alpha}{\beta}$. Denote by $P_{\alpha}$ the total space of the principal $U(1)$-bundle with Chern class $\alpha$ and by $Q_{\beta}$ the total space of the principal $PU$-bundle with Dixmier--Douady invariant $\beta$. Then $(P_{\alpha} \times_X Q_{\beta}) \times_{U(1)^N \times PU^N} \Aut(\OK)$ is a principal $\Aut(\OK)$-bundle over $X$ which corresponds to $\delta$ under the isomorphism of Dadarlat and Pennig.
\end{theorem}

A natural extension of this result would be to consider decomposable classes in higher degrees. For example, we could take $\alpha_1 \cup \alpha_2 \cup \beta \in H^2(X,\ZZ) \cup H^2(X,\ZZ) \cup H^3(X,\ZZ)$ and aim to construct a bundle in much the same way. One could follow the same constructions as above in order to do so, and would obtain an extension of the result.

There are, however, limitations to this approach. It is difficult to determine whether two bundles constructed in this way are isomorphic, and in particular it is even difficult to tell whether a bundle constructed in this way is trivial. For instance, taking a space such as $S^2 \times S^1$ which has non-trivial second- and third-degree integral cohomology, this construction can be used to form a principal $\Aut(\OK)$-bundle over $S^1 \times S^2$. It is not obvious from the construction, however, whether this bundle would correspond to a twist in $H^3(S^1 \times S^2, \ZZ) \cong \ZZ$ or whether it would correspond to an integral 5-class on $S^1 \times S^2$, all of which are trivial.

This motivates a great deal of future research in constructing geometric representatives for twists. Whilst Theorem \ref{decomposable_theorem} provides geometric representatives for a specific class of decomposable twists, the majority of twists cannot be decomposed into pieces as simple as these. Thus it would be desirable to obtain a more general theorem which is applicable to a wider class of decomposable twists, but it is difficult to obtain geometric representatives for higher degree cohomology classes. If models for higher Eilenberg--Mac Lane spaces could be obtained, and injective group homomorphisms from these into $\Aut(\OK)$ could be determined, then the methods used would directly generalise to provide geometric representatives in higher degrees. Alternatively, if results in representing cohomology classes using maps into the stable unitary group could be obtained in some special cases, then the action by outer automorphisms given in Theorem \ref{action} could be used to extend the results presented here. The sections of the constructed bundle can also be explored to determine whether there is any relationship between the sections of the bundle and the sections of the two bundles used in the construction, which would aid in computations.

In general, the problem of associating geometric representatives to cohomology classes is very difficult. A large amount of work has been done on this for classical twists, for instance Brylinski uses the theory of loop groups and transgression of cohomology classes to obtain geometric representatives \cite{Brylinski} and Bouwknegt, Carey, Mathai, Murray and Stevenson produce bundle gerbes representing twists \cite{BCMMS}, but it is not apparent how this work can be carried over to the higher twisted setting. Further research in this area following these ideas may yield more general results.

\section{Spectral sequences}

We develop spectral sequences for higher twisted $K$-theory which will improve our ability to perform computations in the final section. The arguments to obtain the existence of both the Atiyah--Hirzebruch and the Segal spectral sequences for generalised cohomology theories are standard in the literature, for instance in \cite{Eilenberg} and \cite{Segal} respectively. More interesting are results regarding the differentials in these sequences which are unique to higher twisted $K$-theory, and which we will present here. Owing to the difficulty in determining the differentials in these spectral sequences, these results are fairly specialised, but will still be useful in computations.


\subsection{Atiyah--Hirzebruch spectral sequence}

The existence of the Atiyah--Hirzebruch spectral sequence for higher twisted $K$-theory can be established using a filtration of the $K$-theory group determined by the skeletal filtration of the space, and following the standard argument presented in Chapter XV of \cite{Eilenberg}. The full argument is given in Section 4.2.1 of \cite{Brook} and culminates in the following.

\begin{theorem}\label{AHstatement}
Let $X$ be a CW complex with $\delta$ a twist over $X$. There exists an Atiyah--Hirzebruch spectral sequence converging strongly to $K^*(X,\delta)$ with $E^{p,q}_2 = H^p(X,K^q(x_0))$.
\end{theorem}

Note that this spectral sequence is of the exact same form as the Atiyah--Hirzebruch spectral sequence for topological $K$-theory first developed in \cite{AHstart} and that for twisted $K$-theory constructed both by Rosenberg \cite{RosenbergTwist} and by Atiyah and Segal \cite{AS2}. All of these sequences have the same $E_2$-term, but the difference lies in the differentials. In the untwisted case it was found that the first nontrivial differential was given by the Steenrod operation $Sq^3: H^p(X,\ZZ) \to H^{p+3}(X,\ZZ)$, and upon extending to the classical twisted setting the first nontrivial differential became the Steenrod operation twisted by the class $\delta \in H^3(X,\ZZ)$, i.e.\ the differential is expressed by $Sq^3 - (-) \cup \delta: H^p(X,\ZZ) \to H^{p+3}(X,\ZZ)$. We obtain an analogous result in this setting, where we are now forced to restrict to the case that the twist $\delta$ can be represented by a cohomology class. We also lose some information in passing to the higher twisted setting, as the higher differentials of even the Atiyah--Hirzebruch spectral sequence for topological $K$-theory are not well-understood.

\begin{theorem}{\label{HTSS}}
In the setting of the Atiyah--Hirzebruch spectral sequence, if a twist $\delta$ can be represented by $\delta \in H^{2n+1}(X,\ZZ)$ then the $d_{2n+1}$ differential is the differential $d_{2n+1}'$ in the spectral sequence for topological $K$-theory twisted by $\delta$, i.e.\ $d_{2n+1}: H^p(X,\ZZ) \to H^{p+2n+1}(X,\ZZ)$ is given by $d_{2n+1}(x) = d_{2n+1}'(x) - x \cup \delta$.
\end{theorem}

\begin{proof}
We follow the argument given in \cite{AS1}. By definition, the $d_{2n+1}$ differential must be a universal cohomology operation raising degree by $2n+1$, defined for spaces with a given class $\delta \in H^{2n+1}(X,\ZZ)$. Standard arguments in homotopy theory show that these operations are classified by
\[ H^{p+2n+1}(K(\ZZ,p) \times K(\ZZ,2n+1),\ZZ), \]
where the $K(\ZZ,p)$ factor represents cohomology operations raising degree by $2n+1$ and the $K(\ZZ,2n+1)$ factor comes from $X$ being equipped with a class $\delta \in H^{2n+1}(X,\ZZ)$. This cohomology group is isomorphic to
\[ H^{p+2n+1}(K(\ZZ,p),\ZZ) \oplus H^{p+2n+1}(K(\ZZ,2n+1),\ZZ) \oplus \ZZ \]
where the third summand is generated by the product of the generators of $H^{p}(K(\ZZ,p),\ZZ)$ and $H^{2n+1}(K(\ZZ,2n+1),\ZZ)$. The only factor which will actually result in an operation $H^p(X,\ZZ) \to H^{p+2n+1}(X,\ZZ)$ is the first, and so we conclude that the differential is of the form $d_{2n+1}(x) = d_{2n+1}'(x) + k x \cup \delta$ for some $k \in \ZZ$, since the operation must agree with the spectral sequence for topological $K$-theory when $\delta=0$. We determine $k$ by explicitly computing the spectral sequence for $X = S^{2n+1}$ as follows.

The filtration for this case is particularly simple, with $X^0 = X^1 = \cdots = X^{2n}$ each consisting of a single point and $X = S^{2n+1}$. Then the spectral sequence reduces to the long exact sequence for the pair $(X,X^0)$, and the $d_{2n+1}$ differential is the boundary map $K^0(X^0,\delta|_{X^0}) \to K^1(X,X_0;\delta)$. Equivalently, using the excision property of higher twisted $K$-theory applied to the compact pair $(S^{2n+1},D^{2n+1}_+)$ and excising the interior of $D^{2n+1}_+$ we see that $K^1(X,X_0;\delta) \cong K^1(D^{2n+1}_-,S^{2n};\delta|_{D^{2n+1}_-})$ and so $d_{2n+1}$ can be viewed as the boundary map $K^0(D^{2n+1}_+,\delta|_{D^{2n+1}_+}) \to K^1(D^{2n+1}_-,S^{2n};\delta|_{D^{2n+1}_-})$. This map is the passage from top-left to bottom-right in the commutative diagram
\[
\begin{tikzcd}
K^0(D^{2n+1}_+,\delta|_{D^{2n+1}_+}) \ar{r} \ar{d} &K^1(S^{2n+1}, D^{2n+1}_+; \delta) \ar{d} \\
K^0(S^{2n},\delta|_{S^{2n}}) \ar{r} &K^1(D^{2n+1}_-,S^{2n};\delta|_{D^{2n+1}_-}).
\end{tikzcd}
\]
By studying the six-term exact sequence in higher twisted $K$-theory associated to the pair $(D^{2n+1}_-,S^{2n})$, it is clear that the lower horizontal map takes the generator $(1,0)$ of $K^0(S^{2n},\delta|_{S^{2n}}) \cong K^0(S^{2n})$, corresponding to the trivial line bundle over $S^{2n}$, to 0 and the generator $(0,1)$, corresponding to the $n$-fold reduced external product of $(H-1)$ with $H$ the tautological line bundle over $S^{2}$, to the generator of $K^1(D^{2n+1}_-,S^{2n};\delta|_{D^{2n+1}_-})$. All that remains is to determine the left-hand vertical map. This is done in the proof of Proposition \ref{spheres}, in particular this is the top horizontal map in (\ref{CD2}) because in order to identify $K^0(S^{2n},\delta|_{S^{2n}})$ with $\ZZ \oplus \ZZ$ we are using the trivialisation of $\delta$ over $D^{2n+1}_-$ as opposed to $D^{2n+1}_+$. The map is shown to be $n \mapsto (n,-Nn)$ where the twist $\delta \in H^{2n+1}(S^{2n+1},\delta)$ is given by $N \in \ZZ$ times a generator. Hence the composition sends $1 \in \ZZ$ to $-N \in \ZZ$. Since we see that $d_{2n+1}(1) = -N$ then we may conclude that $k = -1$ as required.
\end{proof}

Whilst this is not quite as explicit as the differential in the classical twisted case, since the $d_3$ differential of the Atiyah--Hirzebruch spectral sequence for topological $K$-theory is explicitly known, it is still useful as all differentials in the Atiyah--Hirzebruch spectral sequence for topological $K$-theory are torsion operators \cite{Arlettaz}. Since this result is only applicable when the twist can be represented by cohomology, it will frequently be the case that the space has torsion-free cohomology and so these torsion differentials will have no effect.

Atiyah and Segal are also able to show in \cite{AS1} that the higher differentials of the spectral sequence for classical twisted $K$-theory are given rationally by higher Massey products, and they do so by generalising the Chern character to the twisted setting. This work can likely be generalised to the higher twisted setting, and in fact the Chern character has already been generalised to this setting \cite{MMS}, but since it only gives the differentials rationally it is not highly applicable to computations.

\subsection{Segal spectral sequence}

A more powerful version of the Atiyah--Hirzebruch spectral sequence is the Segal spectral sequence, which we will use for computing higher twisted $K$-theory in more complicated settings. One may work through the details of the construction via a skeletal filtration which induces a filtration of the higher twisted $K$-theory group, but we will not present these details. At this point, we also bring higher twisted $K$-homology back into the picture, because it is in the Segal spectral sequence for higher twisted $K$-homology that the strongest information about the differentials can be easily obtained.

\begin{theorem}{\label{SSS}}
Let $F \xrightarrow{\iota} E \xrightarrow{\pi} B$ be a fibre bundle of CW complexes, and suppose that a twist $\delta$ over $E$ can be represented by a class $\delta \in H^{2n+1}(E,\ZZ)$. Then there is a homological Segal spectral sequence
\[ H_p(B,K_q(F,\iota^* \delta)) \Rightarrow K_*(E,\delta) \]
and a corresponding cohomological Segal spectral sequence
\[ H^p(B,K^q(F,\iota^* \delta)) \Rightarrow K^*(E,\delta). \]
These spectral sequences are strongly convergent if the ordinary (co)homology of $B$ is bounded.
\end{theorem}

\begin{proof}
The proof follows from standard methods, for instance Rosenberg's proof of Theorem 3 in \cite{RosenbergTwist} can be adapted which employs Segal's original proof in Proposition 5.2 of \cite{Segal}.
\end{proof}

\begin{remark}
The ordinary (co)homology of $B$ will be bounded if $B$ is weakly equivalent to a finite-dimensional CW complex and this will cover all of the cases that we consider, so we obtain strong convergence from this spectral sequence.
\end{remark}

Note that we refer to this as a Segal spectral sequence because the method of proof employs Segal's original techniques from \cite{Segal}.

As mentioned above, there is more that can be said about the differentials in the homology spectral sequence. Note that the following theorem uses a Hurewicz map in higher twisted $K$-homology which we have not developed. We will not need to use this map explicitly at any time, and so we do not present the details of its construction. The construction of the map is standard for extraordinary homology theories; see for instance Section II.6 of \cite{Adams} for homology in general and Section II.14 for ordinary $K$-homology. In the higher twisted setting we do not obtain a simple interpretation of the Hurewicz map as Rosenberg is able to for classical twisted $K$-homology in the statement of Theorem 6 \cite{RosenbergTwist}, as $\Aut(\OK)$ is homotopically more complicated than an Eilenberg--Mac Lane space.

\begin{theorem}{\label{differentials}}
In the setting of the homology Segal spectral sequence of Theorem \ref{SSS}, suppose that
\begin{itemize}
\item $\iota^*: H^{2n+1}(E,\ZZ) \to H^{2n+1}(F,\ZZ)$ is an isomorphism, so that the twisting class $\delta$ on $E$ can be identified with the restricted twisting class $\iota^* \delta$ on $F$,
\item the differentials $d^2, \cdots, d^{r-1}$ leave $E^2_{r,0} = H_r(B,K_0(F,\iota^* \delta))$ unchanged, or equivalently $E^2_{r,0} = E^3_{r,0} = \cdots = E^r_{r,0}$, and
\item there is a class $x \in E^2_{r,0}$ which comes from a class $\alpha \in \pi_r(B)$ under the Hurewicz map $\pi_r(B) \to H_r(B,K_0(F,\iota^* \delta))$.
\end{itemize}
Then $d^r(x) \in E^r_{0,r-1}$ is given by the image of $\alpha$ under the composition of the boundary map $\partial: \pi_r(B) \to \pi_{r-1}(F)$ in the long exact sequence of the fibration and the Hurewicz map $\pi_{r-1}(F) \to K_{r-1}(F,\iota^* \delta)$.
\end{theorem}

\begin{proof}
Since the class $x$ was not changed by the differentials $d^2, \cdots, d^{r-1}$ and the twisting class comes from the fibre, we can take $B$ to be $S^r$ and $E = (\RR^r \times F) \cup F$ without loss of generality where $\RR^r \times F$ is $\pi^{-1}$ of the open $r$-cell in $B$. In this special case, as noted by Rosenberg in the proof of Theorem 6 \cite{RosenbergTwist} the spectral sequence comes from the long exact sequence
\[ \cdots \to K_r(F,\iota^* \delta) \xrightarrow{\iota_*} K_r(E,\delta) \to K_r(E,F,\delta) \cong K_0(F,\iota^* \delta) \xrightarrow{\partial} K_{r-1}(F,\iota^* \delta) \to \cdots \]
where we identify $K_0(F,\iota^* \delta)$ with $H_r(B,K_0(F,\iota^* \delta))$. Hence the differential $d^r$ is simply the boundary map in this sequence, and the result follows from the naturality of the Hurewicz homomorphism which implies the commutativity of the diagram
\[
\begin{tikzcd}
\pi_r(B) \ar{r}{\partial} \ar{d}[swap]{\text{Hurewicz}} &\pi_{r-1}(F) \ar{d}{\text{Hurewicz}} \\
H_r(B,K_0(F,\iota^* \delta)) \ar{r}{\partial} &K_{r-1}(F,\iota^* \delta).
\end{tikzcd}
\]
\end{proof}

With these tools developed, we are well-equipped to compute higher twisted $K$-theory for a variety of spaces.

\section{Computations}

This final section is dedicated to computation, allowing for the Mayer--Vietoris sequence and the spectral sequences to be applied. As higher twisted $K$-theory forms a generalisation of both topological and classical twisted $K$-theory, we will see that computations for these variants of $K$-theory will fall out as a result of our computations.

As discussed in the Introduction, computations in higher twisted $K$-theory may be of physical interest in the realms of string theory and M-theory. While we will not give explicit physical descriptions of our computations here, further research into the relationship between higher twisted $K$-theory and physics may help to provide insight into both fields and allow these results to lead to a greater understanding of M-theory.

\subsection{Spheres}


We have an explicit description of the bundles of interest over the spheres via the clutching construction, and so we will begin by computing the higher twisted $K$-theory of the odd-dimensional spheres. This should reduce to known results in the case that trivial twists are used or in the case of classical twists over $S^3$. We will provide a highly detailed computation for the spheres to illustrate the method of using the Mayer--Vietoris sequence and determining the maps in the sequence, and present other computations more concisely.

Note that in this section we refer to ``the generator'' of various cohomology groups isomorphic to $\ZZ$. The choice of generator is unimportant here, because if a twist is $N$ times one generator then it will be $-N$ times the other. The integer $N$ only appears in our results in the form $\ZZ_N$, and since the groups $\ZZ_N$ and $\ZZ_{-N}$ are isomorphic, the results will be true regardless of the choice of generator.

\begin{proposition}\label{spheres}
Let $\delta \in H^{2n+1}(S^{2n+1},\ZZ)$ be a twist of $K$-theory for $S^{2n+1}$ which is $N$ times the generator, where $n \geq 1$. The higher twisted $K$-theory of $S^{2n+1}$ is then
\[ K^0(S^{2n+1}, \delta) = 0 \quad \text{and} \quad K^1(S^{2n+1}, \delta) = \ZZ_N \]
if $N \neq 0$, or
\[ K^0(S^{2n+1}) = \ZZ \quad \text{and} \quad K^1(S^{2n+1}) = \ZZ \]
when $N = 0$.
\end{proposition}

\begin{proof}
Given a twist $\delta$ as above, we take a degree $N$ gluing map $f: S^{2n} \to \Aut(\OK)$ and form the clutching bundle $\cE_f$ to represent $\delta$. We then construct a short exact sequence of $C^*$-algebras including the algebra of sections of the associated algebra bundle $\cA_{\delta}$, allowing the higher twisted $K$-theory groups to be computed via a six-term exact sequence. As shown in Lemma \ref{sections}, the space of sections of this algebra bundle is of the form
\[ A_{\delta} = \{ (h_+,h_-) \in C(D^{2n+1}_+ \amalg D^{2n+1}_-,\OK) : h_+(x) = f(x) \cdot h_-(x) \, \forall \, x \in S^{2n} \}, \]
where again we denote $C(D^{2n+1}_+,\OK) \oplus C(D^{2n+1}_-,\OK)$ by $C(D^{2n+1}_{+} \amalg D^{2n+1}_-,\OK)$ for brevity. Then we may define the short exact sequence
\[ 0 \to A_{\delta} \xrightarrow{\iota} C(D^{2n+1}_{+} \amalg D^{2n+1}_-,\OK) \xrightarrow{\pi} C(S^{2n},\OK) \to 0, \]
where $\iota$ denotes inclusion and $\pi(h_+,h_-)(x) = h_+(x) - f(x) \cdot h_-(x)$ to make the sequence exact. Applying the six-term exact sequence gives \small
\[
\begin{tikzcd}
K_0(A_{\delta}) \ar{r}{\iota_*} &K_0(C(D^{2n+1}_{+} \amalg D^{2n+1}_-,\OK)) \ar{r}{\pi_*} &K_0(C(S^{2n},\OK)) \ar{d}{\partial} \\
K_1(C(S^{2n},\OK)) \ar{u}{\partial} &K_1(C(D^{2n+1}_{+} \amalg D^{2n+1}_-,\OK)) \ar{l}{\pi_*} &K_1(A_{\delta}). \ar{l}{\iota_*}
\end{tikzcd}
\] \normalsize
We are able to simplify several terms in this sequence using trivialisations of the algebra bundle. Firstly, since the hemispheres $D^{2n+1}_+$ and $D^{2n+1}_-$ are contractible, the algebra bundle $\cA_{\delta}$ can be assumed trivial over these hemispheres. To be more precise, using a trivialisation $t_+$ over the upper hemisphere we are able to identify $K_n(C(D^{2n+1}_+,\OK))$ with $\ZZ$ for $n=0$ and $0$ for $n=1$, and similarly trivialising via $t_-$ over the lower hemisphere identifies $K_0(C(D^{2n+1}_-,\OK))$ with $\ZZ$ for $n=0$ and $0$ for $n=1$.

We may also simplify the terms involving the equatorial sphere $S^{2n}$, since the restriction of $\cA_{\delta}$ to $S^{2n}$ will be necessarily trivial due to $S^{2n}$ having trivial odd-degree cohomology. At this point we must make a choice of trivialisation, since we have both $t_+$ and $t_-$ which can trivialise $\cA_{\delta}$ over $S^{2n}$. We choose $t_+$, and in doing so we identify $K_0(S^{2n},\OK)$ with $\ZZ \oplus \ZZ$ for $n=0$ and $0$ for $n=1$.

This reduces the six-term exact sequence to
\[ 0 \to K^0(S^{2n+1},\delta) \to \ZZ \oplus \ZZ \xrightarrow{\pi_*} \ZZ \oplus \ZZ \to K^1(S^{2n+1},\delta) \to 0. \]

The only map to determine here is $\pi_*$, and to do so we must study the differing trivialisations of $\cA_{\delta}$ over $S^{2n}$ since $\pi_*$ is a priori a map between higher twisted $K$-theory groups. Using the Mayer--Vietoris sequence in Proposition \ref{MV}, the map $\pi_*$ is given by the difference $j_{+}^* - j_{-}^*$, where $j_{\pm}: S^{2n} \to D^{2n+1}_{\pm}$ is inclusion and $j_{\pm}^*$ denotes the induced map on higher twisted $K$-theory. Since we have trivialised the bundle over $S^{2n}$ using $t_+$, we will need to take the differing trivialisations into account when determining the map $j_-^*$. The trivialisations of $\cA_{\delta}$ over $S^{2n}$ fit into the commuting diagram
\begin{equation}\label{spherediagram}
\begin{tikzcd}
K^0(S^{2n}, \delta|_{S^{2n}}) \arrow{r}[swap]{(t_+)_*}{\cong}  \arrow{d}{(t_-)_*}[anchor=center,rotate=90,yshift=1ex,swap]{\cong} & K^0(S^{2n}) \\
K^0(S^{2n}) \arrow{ru}&
\end{tikzcd}
\end{equation}
where the map $(t_+)_* \circ {(t_-)_*}^{-1}$ must be determined to change coordinates from $D^{2n+1}_-$ to $D^{2n+1}_+$. In order to do this, we write the trivialisations explicitly.

Firstly, observe that the restriction of $\cA_{\delta}$ to $D^{2n+1}_+$ is the quotient of 
\[ (D^{2n+1}_+ \times (\OK)) \amalg (S^{2n} \times (\OK)) \]
under the usual equivalence relation on the equatorial sphere. Then $t_+$ will be the map sending the class of $(x,o)$ using the representative $x \in D^{2n+1}_+$ to the point $(x,o)$. Similarly, $t_-$ sends the class of $(x,o)$ using the representative $x \in D^{2n+1}_-$ to $(x,o)$. So taking the equivalence relation on $S^{2n}$ into account, these trivialisations differ by the transition function \[t_+ \circ (t_-)^{-1}: S^{2n} \times (\OK) \to S^{2n} \times (\OK)\] given by $(x,o) \mapsto (x, f(x)^{-1}(o))$.



These trivialisations induce maps $(t_{\pm})_*: C(D^{2n+1}_{\pm}, \cA_{\delta}|_{D^{2n+1}_{\pm}}) \to C(S^{2n}, \cA_{\delta}|_{S^{2n}})$ on the section algebras in the obvious way, and composition gives \[(t_+ \circ (t_-)^{-1})_* : C(S^{2n}, \OK) \to C(S^{2n}, \OK)\] sending $g: S^{2n} \to \OK$ to the map $S^{2n} \ni x \mapsto f(x)^{-1} \cdot g(x)$. These maps then in turn induce maps between $K$-theory groups as in the commutative diagram (\ref{spherediagram}).

We will now determine the maps ${j_{\pm}}^*$ induced on higher twisted $K$-theory. Firstly for $j_+$ we have the commutative diagram
\begin{equation}\label{CD1}
\begin{tikzcd}
K^*(D^{2n+1}_+, \delta|_{D^{2n+1}_+}) \ar{r}{{j_+}^*} \ar{d}[swap]{(t_+)_{*}} &K^*(S^{2n}, \delta|_{S^{2n}}) \ar{d}{(t_+)_{*}} \\
K^*(D^{2n+1}_+) \ar{r} &K^*(S^{2n})
\end{tikzcd}
\end{equation}
where the lower map $K^*(D^{2n+1}_+) \to K^*(S^{2n})$ is the map induced by $j_+$ on ordinary $K$-theory. Thus we see that $j^*_+$ is the same as the map induced by $j_+$ on ordinary $K$-theory, which is $j^*_+(m)=(m,0)$. This is expected, because both bundles are trivialised via $t_+$. On $D^{2n+1}_-$, however, we must change coordinates via the transition function $t_+ \circ (t_-)^{-1}$ so that we are trivialising the bundle over $S^{2n}$ via $t_+$ rather than $t_-$. This gives the diagram
\begin{equation}\label{CD2}
\begin{tikzcd}
K^*(D^{2n+1}_-, \delta|_{D^{2n+1}_-}) \ar{r}{{j_-}^*} \ar{d}[swap]{(t_-)_{*}} &K^*(S^{2n}, \delta|_{S^{2n}}) \ar{d}{(t_+)_{*}} \\
K^*(D^{2n+1}_-) \ar{r} &K^*(S^{2n})
\end{tikzcd}
\end{equation}
and thus $j_-^*$ can be viewed as the map induced by $j_-$ on ordinary $K$-theory followed by $(t_+ \circ (t_-)^{-1})_{*}$. Since the map induced by $t_+ \circ (t_-)^{-1}$ on section algebras is multiplication by $f^{-1}$, we seek the map on $K$-theory induced by the composition
\[ C(D^{2n+1}_-,\cA_{\delta}|_{D^{2n+1}_-}) \xrightarrow{\text{res}} C(S^{2n},\cA_{\delta}|_{S^{2n}}) \xrightarrow{\times f^{-1}} C(S^{2n},\cA_{\delta}|_{S^{2n}}). \]
In the case of topological $K$-theory when $N = 0$ this is the map $n \mapsto (n,0)$, but if $N \neq 0$ then the second component of this map is non-trivial, resulting in $n \mapsto (n,-Nn)$ with the factor of $-N$ corresponding to multiplication by $f^{-1}$.

Thus $\pi_*(m,n) = (m,0) - (n,-Nn) = (m-n,Nn)$, which has trivial kernel and whose cokernel is $(\ZZ \oplus \ZZ) / (\ZZ \oplus N \ZZ) \cong \ZZ_N$ when $N \neq 0$. So we are able to conclude via the exact sequence that $K^0(S^{2n+1},\delta) = 0$ while $K^1(S^{2n+1},\delta) \cong \ZZ_N$. Note that if $N=0$ we instead have $\pi_*(m,n) = (m-n,0)$ with kernel and cokernel $\ZZ$ corresponding to the standard topological $K$-theory of $S^{2n+1}$.
\end{proof}

While this computation shows that the higher twisted $K$-theory of the odd-dimensional spheres agrees with and extends the classical notion of twisted $K$-theory for $S^3$, it is desirable to have an explicit geometric representative for the generator of $K^1(S^{2n+1},\delta)$. To do so, we shift our viewpoint to the equivalent definition of higher twisted $K$-theory in terms of generalised Fredholm operators presented in Theorem \ref{topological}. Firstly, we need a lemma allowing us to view this higher twisted $K$-theory group in a slightly different way.

\begin{lemma}\label{K1sphere}
The higher twisted $K$-theory group $K^1(S^{2n+1},\delta)$ can be expressed as
\[ \pi_0(\{ (h_+, h_-) \in C(D^{2n+1}_+ \amalg D^{2n+1}_-, \Omega \Fred_{\OK})  : h_+(x) = f(x) \cdot h_-(x) \, \forall \, x \in S^{2n} \}). \]
\end{lemma}


\begin{proof}
Recall that $K^1(S^{2n+1},\delta) = \pi_0(C(\cE_{\delta},\Omega \Fred_{\OK})^{\Aut(\OK)})$. Since $\cE_{\delta}$ is constructed via the clutching construction, an element of $C(\cE_{\delta},\Omega \Fred_{\OK})^{\Aut(\OK)}$ can be viewed as a pair of maps $h_{\pm}: D^{2n+1}_{\pm} \to \Omega \Fred_{\OK}$ satisfying $h_+(x) = f(x) \cdot h_-(x)$ for all $x \in S^{2n}$. Conversely, any such pair of maps $h_{\pm}$ may be glued together to form the equivariant map $h: \cE_{\delta} \to \Omega \Fred_{\OK}$ via
\[ h([x,T])(t)(v) = \begin{cases} h_+(x)(t)(T \cdot v) &\text{ if } x \in D^{2n+1}_+; \\ h_-(x)(t)(T \cdot v) &\text{ if } x \in D^{2n+1}_-; \end{cases} \]
where $x \in S^{2n+1}$, $T \in \Aut(\OK)$, $t \in S^1$ and $v \in \cH_{\OK}$. Firstly, we note that $h$ is well-defined. If $[(x,T)]$ is chosen with $x \in S^{2n}$, then the two possible representatives of this point are $(x, T) \in D^{2n+1}_+ \times \Aut(\OK)$ and $(x,f(x) \circ T) \in D^{2n+1}_- \times \Aut(\OK)$. But $h_-(x)(t)((f(x) \circ T) \cdot v) = (f(x) \cdot h_-(x))(t)(T \cdot v)$ by the definition of the action, and this is equal to $h_+(x)(t)(T \cdot v)$ by the compatibility of the maps on the equatorial sphere. Similarly, the map $h$ is $\Aut(\OK)$-equivariant by the definition of the action as required.
\end{proof}

Using this viewpoint, we are able to construct a representative for the generator of the higher twisted $K$-theory group. In our description of the generator, we use an isomorphism between the hemisphere $D^{2n+1}_+$ with its boundary identified to a point and the sphere $S^{2n+1}$. This allows us to view a map on $S^{2n+1}$ as a map on $D^{2n+1}_+$ which is constant on the boundary. We illustrate this for clarity in Figure \ref{figure} in the case of $S^2$ which can be visualised.

\begin{figure}[!htb]
\hspace{-0.8cm}
    \begin{minipage}{0.25\textwidth}
      \includegraphics[scale=2]{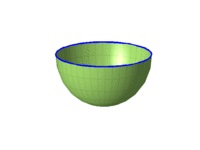}
     \end{minipage} \qquad \quad $\cong$
     \begin{minipage}{0.25\textwidth}
        \includegraphics[scale=2]{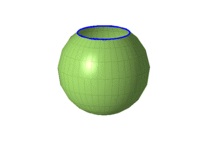}
     \end{minipage} \qquad \quad $\cong$
     \begin{minipage}{0.25\textwidth}
       \includegraphics[scale=2]{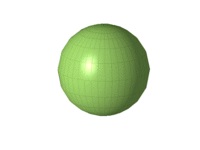}
     \end{minipage}
  \caption{Illustration of the isomorphism between the hemisphere $D^2_+$ with its boundary identified to a point and the 2-sphere $S^2$. Adapted from \cite{Image}.} \label{figure}
\end{figure}

\begin{proposition}\label{generator}
The generator of $K^1(S^{2n+1},\delta)$ can be represented by the pair of maps $h_{\pm}$ where $h_+$ is obtained by taking the generator $k \in \pi_{2n+1}(\Omega \Fred_{\OK})$ and viewing this as a map $D^{2n+1}_+ \to \Omega \Fred_{\OK}$ which is constant on the equatorial sphere via the isomorphism displayed in Figure \ref{figure}, and $h_-$ is a loop which remains constant at the identity operator.
\end{proposition}

\begin{proof}
In order to obtain a generator, we use a different short exact sequence of $C^*$-algebras to obtain a six-term exact sequence in $K$-theory. Here we take the sequence
\[ 0 \to C_0(\RR^{2n+1},\OK) \xrightarrow{\iota} A_{\delta} \xrightarrow{\pi} C(x_0, \OK) \to 0 \]
for $x_0 \in S^{2n+1}$ defined so that $S^{2n+1} \setminus \{x_0\} \cong \RR^{2n+1}$, with the obvious maps for $x_0 \notin S^{2n}$. Note that we may take sections of the trivial bundle over $\RR^{2n+1}$ since there are no non-trivial principal $\Aut(\OK)$-bundles over $\RR^{2n+1}$, and similarly for $\{x_0\}$. This gives rise to the six-term exact sequence
\[
\begin{tikzcd}
0 = K^0(\RR^{2n+1}) \ar{r}{\iota_*} &K^0(S^{2n+1},\delta) \ar{r}{\pi_*} &K^0(\{x_0\}) \ar{d}{\partial} = \ZZ \\
0 = K^1(\{x_0\}) \ar{u}{\partial} &K^1(S^{2n+1},\delta) \ar{l}{\pi_*} &K^1(\RR^{2n+1}) = \ZZ, \ar{l}{\iota_*}
\end{tikzcd}
\]
where twisted $K$-theory groups equipped with the trivial twisting have been identified with their untwisted counterparts, and this reduces to
\[0 \to \ZZ \xrightarrow{\partial} \ZZ \xrightarrow{\iota_*} K^1(S^{2n+1},\delta) \to 0. \]
By Proposition \ref{spheres} we know that $K^1(S^{2n+1},\delta) = \ZZ_N$ and so $\iota_*$ is a surjective map from $\ZZ$ to $\ZZ_N$. This means that it must be given by reduction modulo $N$ and hence the generator of $K^1(S^{2n+1},\delta)$ is the image of the generator of $K^1(\RR^{2n+1}) \cong \ZZ$ under $\iota_*$. The map $\iota_*$ can be interpreted by making the following identifications:
\begin{align*}
K^1(\RR^{2n+1}) &= \widetilde{K}^1(S^{2n+1}) \\
&\cong K^1(S^{2n+1}) \\
&\cong [S^{2n+1},\Omega \Fred_{\OK}] \\
&= \pi_{2n+1}(\Omega \Fred_{\OK}),
\end{align*}
where we use the fact that the ordinary topological $K$-theory of $S^{2n+1}$ is the same as the higher twisted $K$-theory of $S^{2n+1}$ with the trivial twist.

In order to realise the reduction modulo $N$ map from $\pi_{2n+1}(\Omega \Fred_{\OK})$ into $K^1(S^{2n+1},\delta)$, we let $[k: S^{2n+1} \to \Omega \Fred_{\OK}] \in \pi_{2n+1}(\Omega \Fred_{\OK})$ be the generator and by identifying $S^{2n+1}$ with $D^{2n+1}_+ /\sim$ as illustrated in Figure \ref{figure}, we view $k$ as a map $h_+$ on $D^{2n+1}_+$ which is constant at the identity on the equatorial sphere. Then defining a map $h_-$ on $D^{2n+1}_-$ to be a loop which is constant at the identity gives a pair $[h_{\pm}] \in K^1(S^{2n+1},\delta)$ via Lemma \ref{K1sphere}. Applying this process with $M$ times the generator of $\pi_{2n+1}(\Omega \Fred_{\OK})$ yields an element of $K^1(S^{2n+1},\delta)$ which is $M$ mod $N$ times the generator. Thus the generator of $K^1(S^{2n+1},\delta)$ is obtained by applying this process to the generator $k$ itself as required.
\end{proof}

Note that the choice of generator of $\pi_{2n+1}(\Omega \Fred_{\OK})$ is once again unimportant here, and different choices will simply yield different generators of $\ZZ_N$.

This formulation of the generator agrees with and extends that given by Mickelsson in the classical twisted setting \cite{Mickelsson}, and the existence of such an explicit generator in terms of the generator of $\pi_{2n+1}(\Omega \Fred_{\OK})$ may have a physical interpretation which could be used to further investigate relevant areas of physics.

It should be noted that obtaining explicit expressions for the generators of higher twisted $K$-theory groups is difficult in general. In this case we relied on applying the Mayer--Vietoris sequence as well as a useful identification of a topological $K$-theory group with a homotopy group. This will not be possible in other cases, leaving potential for future work in finding more general methods to express the generators of higher twisted $K$-theory groups.

To complete our computations for odd-dimensional spheres, we provide a more straightforward proof of Proposition \ref{spheres} using the twisted Atiyah--Hirzebruch spectral sequence. Recall that the Mayer--Vietoris proof of Proposition \ref{spheres} was used in the proof of Theorem \ref{HTSS}; this example is simply to illustrate the use of the spectral sequence rather than to provide an alternative proof.

\begin{example}
We use the spectral sequence in Theorem \ref{AHstatement}. Since $H^p(S^{2n+1},\ZZ) \cong \ZZ$ if $p=0$ or $2n+1$ and the cohomology is trivial otherwise, we see that $E^{p,q}_2 \cong E^{p,q}_r$ for $2 \leq r \leq 2n$. The only non-zero differential is then $d_{2n+1}: H^0(S^{2n+1},\ZZ) \to H^{2n+1}(S^{2n+1},\ZZ)$, as displayed on the left side of Figure \ref{sphereSS}.
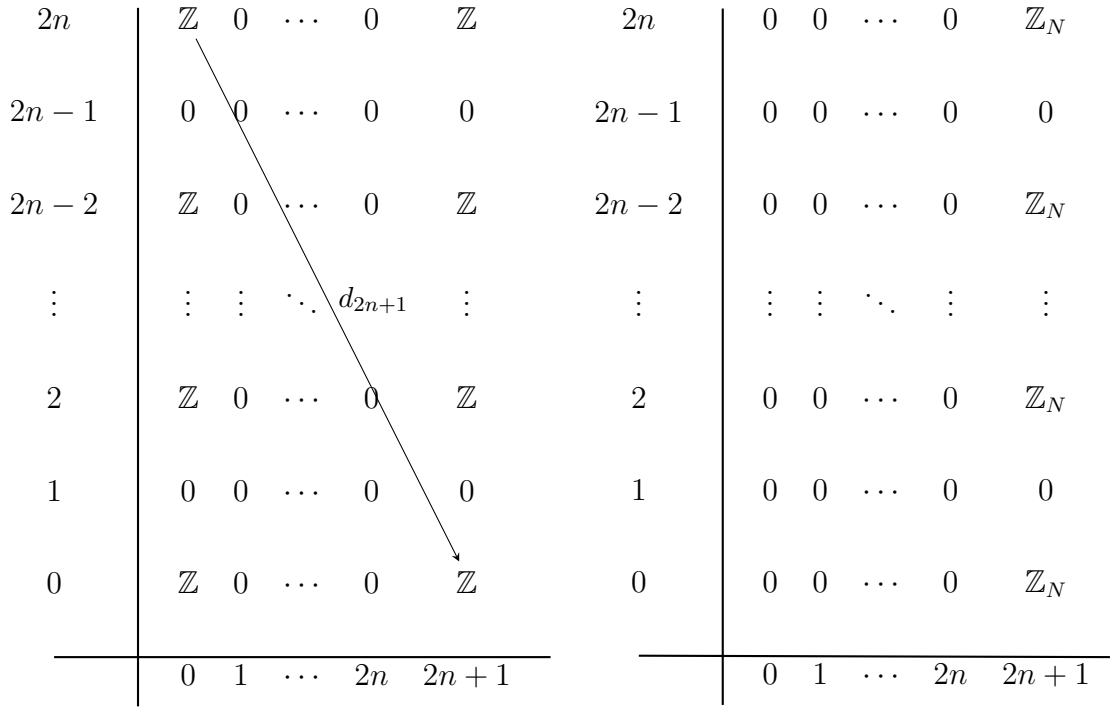
\begin{figure}
\begin{center}
\begin{subfigure}{0.45\linewidth}
\begin{tikzpicture}
  \matrix (m) [matrix of math nodes,
    nodes in empty cells,nodes={minimum width=2.3ex,
    minimum height=5ex,outer sep=-5pt},
    column sep=1ex,row sep=1ex]{
		2n    	& &	\ZZ  & 0 & \cdots & 0  & \ZZ & \\
		2n-1     	& &	0 &  0  & \cdots & 0 & 0 & \\
		2n-2    	& &	\ZZ  & 0 & \cdots & 0  & \ZZ & \\
    		\vdots	& &	\vdots & \vdots & \ddots &  & \vdots & \\
		2    	& &	\ZZ  & 0 & \cdots & 0  & \ZZ & \\
		1     	& &	0 &  0  & \cdots & 0 & 0 & \\
		0    	& &	\ZZ  & 0 & \cdots & 0  & \ZZ & \\
    \quad\strut & &  0  &  1  &  \cdots & 2n & 2n+1 & \strut \\};
  \draw[-stealth] (m-1-3.south east) -- node[above,right] {\small{$d_{2n+1}$}} (m-7-7.north west);
\draw[thick] (m-1-2.north) -- (m-8-2.south) ;
\draw[thick] (m-8-1.north) -- (m-8-8.north) ;
\end{tikzpicture}
\end{subfigure}
\,
\begin{subfigure}{0.45\linewidth}
\begin{tikzpicture}
  \matrix (m) [matrix of math nodes,
    nodes in empty cells,nodes={minimum width=2.3ex,
    minimum height=5ex,outer sep=-5pt},
    column sep=1ex,row sep=1ex]{
		2n    	&&0  & 0 & \cdots & 0  & \ZZ_N \\
		2n-1     	&&0 &  0  & \cdots & 0 & 0 \\
		2n-2    	&&0  & 0 & \cdots & 0  & \ZZ_N \\
    		\vdots	&&\vdots & \vdots & \ddots & \vdots & \vdots \\
		2    	&&0  & 0 & \cdots & 0  & \ZZ_N \\
		1     	&&0 &  0  & \cdots & 0 & 0 \\
		0    	&&0  & 0 & \cdots & 0  & \ZZ_N \\
    \strut &&  0  &  1  &  \cdots & 2n & 2n+1 \strut \\};
\draw[thick] (m-1-2.north) -- (m-8-2.south) ;
\draw[thick] (m-8-1.north) -- (m-8-8.north) ;
\end{tikzpicture}
\end{subfigure}
\end{center} \caption{Atiyah--Hirzebruch spectral sequence for $S^{2n+1}$.} \label{sphereSS}
\end{figure}
By Theorem \ref{HTSS}, this differential is given by $d_{2n+1}(x) = d_{2n+1}'(x) - x \cup \delta$ where $d_{2n+1}'$ is some torsion operator, i.e.\ the image of $d_{2n+1}'$ is torsion. Thus the differential is simply cup product with $-\delta$, meaning that the $E_{2n+1} \cong \cdots \cong E_{\infty}$ term is as shown on the right of Figure \ref{sphereSS}. Then by the convergence of the spectral sequence, we may once again conclude that $K^0(S^{2n+1},\delta)= 0$ while $K^1(S^{2n+1},\delta) \cong \ZZ_N$.
\end{example}

While this computation is much more manageable, it cannot be used to obtain any information about the generators of the higher twisted $K$-theory groups. Due to the simplicity of this method, however, it will prove useful in computing the higher twisted $K$-theory of more complicated spaces.

\subsection{Products of spheres}

We may apply the same techniques to compute the higher twisted $K$-theory of products of spheres. In theory, it is possible to consider any product of spheres consisting of at least one odd-dimensional sphere, since the product will have torsion-free cohomology and non-trivial cohomology in at least one odd degree. In practice, however, without developing more general techniques we are limited to a smaller class of products.

One such product that we can compute via the same methods as used to prove Proposition \ref{spheres} is $S^{2m} \times S^{2n+1}$ for $m, n \geq 1$. This space has non-trivial odd-degree cohomology groups in degrees $2n+1$ and $2m+2n+1$, both of which are isomorphic to $\ZZ$. The clutching construction is not applicable to twists coming from $(2m+2n+1)$-classes, as these would produce bundles over $S^{2m+2n+1}$ as opposed to $S^{2m} \times S^{2n+1}$. In spite of this, a modified version of the clutching construction can be used for twists of degree $2n+1$. We take trivial bundles over $S^{2m} \times D^{2n+1}_{\pm}$ and modify the gluing map $[f] \in \pi_{2n}(\Aut(\OK))) \cong H^{2n+1}(S^{2n+1},\ZZ)$ to a map $\widetilde{f}: S^{2m} \times S^{2n} \to \Aut(\OK)$ which is constant over the $S^{2m}$ factor, i.e.\ $\widetilde{f}(x,y) = f(y)$. The resulting bundle constructed using this gluing map will then pull back to the trivial bundle over $S^{2m}$ and to the usual clutching bundle associated to $f$ over $S^{2n+1}$. Using this explicit bundle, the Mayer--Vietoris sequence can be used to compute the higher twisted $K$-theory groups. We will not provide this Mayer--Vietoris proof; in this case the proof is almost identical to the proof of Proposition \ref{spheres}. For the full details, see the proof of Proposition 5.1.4 in \cite{Brook}. Instead, we use a K\"unneth theorem in $C^*$-algebraic $K$-theory to illustrate another method of computation.

\begin{proposition}\label{products}
Let $\delta \in H^{2n+1}(S^{2m} \times S^{2n+1},\ZZ)$ be a twist of $K$-theory for $S^{2m} \times S^{2n+1}$ which is $N$ times the generator, where $m, n \geq 1$. The higher twisted $K$-theory of $S^{2m} \times S^{2n+1}$ is then
\[ K^0(S^{2m} \times S^{2n+1}, \delta) = 0 \quad \text{and} \quad K^1(S^{2m} \times S^{2n+1},\delta) = \ZZ_N \oplus \ZZ_N \]
if $N \neq 0$, or
\[ K^0(S^{2m} \times S^{2n+1}) = \ZZ \oplus \ZZ \quad \text{and} \quad K^1(S^{2m} \times S^{2n+1}) = \ZZ \oplus \ZZ \]
when $N=0$.
\end{proposition}

\begin{proof}
We use the K\"unneth theorem given in Theorem 23.1.3 of \cite{BlackK}, which states that
\[ 0 \to K_*(A) \otimes K_*(B) \to K_*(A \otimes B) \to \Tor_1^{\ZZ}(K_*(A), K_*(B)) \to 0 \]
is a short exact sequence if $A$ belongs in the bootstrap category of $C^*$-algebras defined in Definition 22.3.4 of \cite{BlackK}. We let $A$ denote the continuous complex-valued functions on $S^{2m}$ -- which belongs in the bootstrap category as do all commutative $C^*$-algebras -- so that $K_*(A) = K^*(S^{2m})$, and we take $B$ to be the algebra of sections of the algebra bundle over $S^{2n+1}$ representing the twist $\delta$ so that $K_*(B) = K^*(S^{2n+1},\delta)$. Since $K_*(A)$ is torsion-free, the $\Tor_1$ term will be trivial and thus we obtain an isomorphism $K_*(A) \otimes K_*(B) \cong K_*(A \otimes B)$. Now, since the algebra bundle $\cA_{\widetilde{f}}$ is trivial over the factor of $S^{2m}$, the sections of the bundle can be split into
\[ C(S^{2m} \times S^{2n+1}, \cA_{\widetilde{f}}) = C(S^{2m},\OK) \otimes C(S^{2n+1},\cA_{\widetilde{f}}|_{S^{2n+1}}) \cong C(S^{2m}) \otimes C(S^{2n+1},\cA_{f}). \]
Therefore $A \otimes B$ is isomorphic to the space of sections of $\cA_{\widetilde{f}}$, and so we conclude that $K_*(A \otimes B) = K^*(S^{2m} \times S^{2n+1},\delta)$. Hence the isomorphism given by the K\"unneth theorem verifies that
\[ K^0(S^{2m} \times S^{2n+1},\delta) \cong ((\ZZ \oplus \ZZ) \otimes 0) \oplus (0 \otimes \ZZ_N) = 0 \]
and
\[ K^1(S^{2m} \times S^{2n+1},\delta) \cong ((\ZZ \oplus \ZZ) \otimes \ZZ_N) \oplus (0 \otimes 0) = \ZZ_N \oplus \ZZ_N \]
when $N \neq 0$ as required. For the case $N=0$ the K\"unneth theorem can be applied directly to the topological $K$-theory groups of $S^{2m}$ and $S^{2n+1}$.
\end{proof}

Note that in this case we have two different generators of order $N$ for the $K^1$-group, and so it would be of interest to explicitly write down these generators. The Mayer--Vietoris technique used for $S^{2n+1}$, however, does not generalise to this case and so this would require the development of further machinery.

Although twists of degree $2m+2n+1$ do not have such a simple geometric interpretation, the computation can easily be carried out using the Atiyah--Hirzebruch spectral sequence to obtain the following.

\begin{proposition}\label{products2}
Let $\delta \in H^{2m+2n+1}(S^{2m} \times S^{2n+1},\ZZ)$ be a higher twist of $K$-theory for $S^{2m} \times S^{2n+1}$ which is $N$ times the generator, where $m, n \geq 1$. The higher twisted $K$-theory of $S^{2m} \times S^{2n+1}$ is then
\[ K^0(S^{2m} \times S^{2n+1}, \delta) = \ZZ \quad \text{and} \quad K^1(S^{2m} \times S^{2n+1},\delta) = \ZZ \oplus \ZZ_N \]
if $N \neq 0$, or
\[ K^0(S^{2m} \times S^{2n+1}) = \ZZ \oplus \ZZ \quad \text{and} \quad K^1(S^{2m} \times S^{2n+1}) = \ZZ \oplus \ZZ \]
when $N=0$.
\end{proposition}

The proof is a straightforward application of Theorem \ref{HTSS} with no extension problems. We note that this computation cannot be verified using the K\"unneth theorem because we cannot express the space of sections of the algebra bundle representing $\delta \in H^{2m+2n+1}(S^{2m} \times S^{2n+1},\ZZ)$ as a tensor product as we were able to in the previous case. Suppose we were to decompose $\delta$ into the cup product of $\delta_{2m} \in H^{2m}(S^{2m},\ZZ)$ and $\delta_{2n+1} \in H^{2n+1}(S^{2n+1},\ZZ)$ via a K\"unneth theorem in cohomology, and then we let $A$ be the space of sections of the bundle over $S^{2m}$ represented by $\delta_{2m}$ and $B$ be the space of sections of the bundle over $S^{2n+1}$ represented by $\delta_{2n+1}$. Then $A$ and $B$ would be exactly as in the proof of Proposition \ref{products}, since there are no non-trivial algebra bundles over $S^{2m}$ with fibres isomorphic to $\OK$, and so the tensor product algebra $A \otimes B$ would remain unchanged.

Of course there are many other possible products of spheres that can be investigated, and the spectral sequence can be used in a straightforward way to draw conclusions about the higher twisted $K$-theory groups. In spite of this, most cases involve non-trivial extension problems and so it is difficult to obtain results for products of spheres in full generality using current techniques.

\subsection{Special unitary groups}

Another particularly useful class of spaces with torsion-free cohomology is formed by certain types of Lie groups. A great deal of work has been done by many mathematicians and physicists in computing the twisted $K$-theory of Lie groups in the classical setting, including Hopkins, Braun \cite{Braun}, Douglas \cite{Douglas}, Rosenberg \cite{RosenbergLie} and Moore \cite{Moore}. In the case of $SU(n)$, the twisted $K$-groups were explicitly computed and as a consequence it was shown that the higher differentials in the twisted Atiyah--Hirzebruch spectral sequence are non-zero in general, suggesting that this technique will not yield general results for the higher twisted $K$-groups of $SU(n)$. Nevertheless, it is possible to compute these groups via the Atiyah--Hirzebruch spectral sequence in a special case, and to use this in more general computations.

We compute the higher twisted $K$-theory of $SU(n)$ up to extension problems for $\delta$ a $(2n-1)$-twist. Although this does not fully describe the higher twisted $K$-theory groups, it gives important information regarding torsion and the maximum order of elements in the groups. To illustrate the general technique, we will start by computing the higher twisted $K$-theory of $SU(3)$ for a $5$-twist. Note that to simplify the statements of results we will henceforth only consider non-trivial twists, as the trivial case reduces to topological $K$-theory.

\begin{lemma}\label{SU(3)}
Let $\delta \in H^{5}(SU(3),\ZZ)$ be a twist of $K$-theory for $SU(3)$ which is $N \neq 0$ times the generator. The 5-twisted $K$-theory of $SU(3)$ is then
\[ K^0(SU(3), \delta) = \ZZ_N \quad \text{and} \quad K^1(SU(3),\delta) = \ZZ_N. \]
\end{lemma}

\begin{proof}
The $E_2$-page of the twisted Atiyah--Hirzebruch spectral sequence is as follows.
\begin{center}
\begin{tikzpicture}
  \matrix (m) [matrix of math nodes,
    nodes in empty cells,nodes={minimum width=2.3ex,
    minimum height=5ex,outer sep=-5pt},
    column sep=1ex,row sep=1ex]{
		2    	& &	\ZZ  & 0 & 0 & \ZZ & 0 & \ZZ & 0 & 0 & \ZZ &\\
		1     	& &   0 &  0  & 0 & 0 & 0 & 0 & 0 & 0 & 0 &\\
		0    	& &	\ZZ  & 0 & 0 & \ZZ & 0 & \ZZ & 0 & 0 & \ZZ &\\
     \strut & &  0  &  1  &  2 & 3 & 4 & 5 & 6 & 7 & 8 & \strut \\};
\draw[thick] (m-1-2.north) -- (m-4-2.south) ;
\draw[thick] (m-4-1.north) -- (m-4-12.north) ;
\end{tikzpicture}
\end{center}

The $d_3$ differential is given by the Steenrod operation $Sq^3$ which necessarily annihilates $H^0(SU(3),\ZZ)$ by Theorem 4L.12 of \cite{HatcherAT} and also annihilates $H^5(SU(3),\ZZ)$ since the image of $Sq^3$ is a 2-torsion element by definition. Hence the only non-trivial differential in this spectral sequence is $d_5(x) = d_5'(x) - x \cup \delta$. The torsion operator $d_5'$ will have no effect on the cohomology, and cup product with $-\delta$ will be multiplication by $-N$ on both $H^0(SU(3),\ZZ)$ and $H^3(SU(3),\ZZ)$. Hence the $E_{\infty}$-term is as shown below.
\begin{center}
\begin{tikzpicture}
  \matrix (m) [matrix of math nodes,
    nodes in empty cells,nodes={minimum width=2.3ex,
    minimum height=5ex,outer sep=-5pt},
    column sep=1ex,row sep=1ex]{
		2    	& &	0  & 0 & 0 & 0 & 0 & \ZZ_N & 0 & 0 & \ZZ_N &\\
		1     	& &   0 &  0  & 0 & 0 & 0 & 0 & 0 & 0 & 0 &\\
		0    	& &	0  & 0 & 0 & 0 & 0 & \ZZ_N & 0 & 0 & \ZZ_N &\\
     \strut & &  0  &  1  &  2 & 3 & 4 & 5 & 6 & 7 & 8 & \strut \\};
\draw[thick] (m-1-2.north) -- (m-4-2.south) ;
\draw[thick] (m-4-1.north) -- (m-4-12.north) ;
\end{tikzpicture}
\end{center}
Thus we may conclude that $K^0(SU(3),\delta) \cong \ZZ_N$ and $K^1(SU(3),\delta) \cong \ZZ_N$ as required.
\end{proof}

Note that this computation for $SU(3)$ works specifically for $5$-twists $\delta$, as unlike when taking a classical $3$-twist there are no non-trivial higher differentials to consider. Furthermore, there are no extension problems to solve and so this is a complete computation. The method of computation directly generalises to the case of $(2n-1)$-twists on $SU(n)$, although here we only obtain the result up to extension problems and so we can only comment on torsion in the group.

\begin{lemma}\label{SU(n)}
Let $\delta \in H^{2n-1}(SU(n),\ZZ)$ be a twist of $K$-theory for $SU(n)$ which is $N \neq 0$ times the generator. The $(2n-1)$-twisted $K$-theory of $SU(n)$ is then a finite abelian group with all elements having order a divisor of a power of $N$.
\end{lemma}

\begin{proof}
We use the same Atiyah--Hirzebruch spectral sequence approach as in the proof of Lemma \ref{SU(3)}.
The differentials $d_j$ for $j < 2n-1$ are trivial, as they are given by torsion operations. The differential $d_{2n-1}$ is cup product with $-\delta$, which is multiplication by $-N$ for each of the $2^{n-2}$ maps $\ZZ (\bigwedge c_i) \to \ZZ (\bigwedge c_i) \wedge c_{2n-1}$ where the $c_{2i-1} \in H^{2i-1}(SU(n), \ZZ)$ for $i = 2, \cdots, n$ denote the primitive generators, and the higher differentials are zero. At this stage, there are $2^{n-3}$ extension problems to solve for $n>3$, but no extension problems for $n=3$ which is how the previous result was obtained. In spite of this, since every group in the $E_{\infty}$-term of the spectral sequence is $\ZZ_N$, we can conclude that the higher twisted $K$-theory groups will be torsion with all elements having order a divisor of a power of $N$, even if the extension problems cannot be solved to determine the explicit torsion.
\end{proof}

To be more explicit about the extension problems involved, we consider the case of a $7$-twist on $SU(4)$.

\begin{example}
Following the proof of Lemma \ref{SU(n)}, we need to solve a single extension problem both for the odd-degree and even-degree groups of the form
\[ 0 \to \ZZ_N \to K^i(SU(4),\delta) \to \ZZ_N \to 0. \]
Although this extension problem has $\text{Ext}^1_{\ZZ}(\ZZ_N,\ZZ_N) \cong \ZZ_N$ inequivalent solutions, we can conclude that $K^i(SU(4),\delta)$ is a torsion group whose elements have order a divisor of $N^2$.
\end{example}

While an explicit computation of the higher twisted $K$-theory of $SU(n)$ is difficult in general, we employ techniques of Rosenberg \cite{RosenbergLie} to obtain structural information about these groups. It is at this point that we must turn to the more powerful Segal spectral sequence given in Theorem \ref{SSS} so that Theorem \ref{differentials} may be employed. We begin with a result specifically for 5-twists, before presenting a more general result.

\begin{remark}
We reiterate the concerns raised in Remark \ref{bootstrap} that the proof of Proposition \ref{K-homology} relies on an assumption which is based on a conjecture in $C^*$-algebra theory. Theorem \ref{concerned} should be viewed in light of this assumption.
\end{remark}

\begin{theorem}\label{concerned}
For any non-zero $\delta \in H^{5}(SU(n+1),\ZZ)$ given by $N$ times the generator with $N$ relatively prime to $n!$ ($n>1$), the graded group $K^*(SU(n+1),\delta)$ is isomorphic to $\ZZ_N$ tensored with an exterior algebra on $n-1$ odd generators.
\end{theorem}

\begin{proof}
We proceed by induction on $n$. First, note that the case $n=2$ has already been proved in Lemma \ref{SU(3)}, as we have shown that $K^0(SU(3),\delta) \cong \ZZ_N \cong K^1(SU(3),\delta)$ so that $K^*(SU(3),\delta)$ is of the form $\ZZ_N$ tensored with $\ZZ x$ for some odd generator $x$. Then by Proposition \ref{K-homology}, the same is true for the higher twisted $K$-homology groups. So assume $n > 2$ and that the result holds for higher twisted $K$-homology for smaller values of $n$. Take the Segal spectral sequence for higher twisted $K$-homology associated to the classical fibration
\[ SU(n) \xhookrightarrow{\iota} SU(n+1) \to S^{2n+1}, \]
which gives
\[ E^2_{p,q} = H_p(S^{2n+1},K_q(SU(n),\iota^* \delta)) \Rightarrow K_*(SU(n+1),\delta). \]
Note that since the map $\iota^*$ induced on ordinary cohomology by inclusion is an isomorphism in degree 5, we may identify $\iota^* \delta \in H^5(SU(n),\ZZ)$ with $\delta \in H^5(SU(n+1),\ZZ)$. Since $N$ is relatively prime to $(n-1)!$, by the inductive assumption we have \[K_*(SU(n),\delta) \cong \ZZ_N \otimes \wedge(x_1, \cdots, x_{n-2})\] for some odd generators $x_i$. We aim to show that the spectral sequence collapses. The only potentially non-zero differential is $d^{2n+1}$, which is related to the homotopical non-triviality of the fibration as explained in Theorem \ref{differentials}. To determine the explicit differential, we need to understand the Hurewicz maps mentioned in the theorem and the long exact sequence in homotopy for the fibration $SU(n) \xhookrightarrow{\iota} SU(n+1) \to S^{2n+1}$. This long exact sequence contains
\[ \pi_{2n+1}(SU(n+1)) \to \pi_{2n+1}(S^{2n+1}) \xrightarrow{\partial} \pi_{2n}(SU(n)) \to \pi_{2n}(SU(n+1)), \]
and so we see that the boundary map $\partial: \ZZ \to \ZZ_{n!}$ has kernel of index $n!$. Now, the Hurewicz map of interest is 
\[ \pi_{2n}(SU(n)) \to K_{2n}(SU(n),\delta) \cong K_0(SU(n),\delta). \]
Although this map is difficult to describe explicitly, since it is a map from $\ZZ_{n!}$ to a group generated by elements of order $N$ then if $\gcd(N,n!) = 1$ this map must be trivial and hence the differential is trivial. Thus if $\gcd(N,n!) = 1$ then the spectral sequence collapses and $K_*(SU(n+1), \delta)$ is isomorphic to $\ZZ_N$ tensored with an exterior algebra on $n-1$ odd generators since the $E^{\infty}$-term of the spectral sequence will consist of $\ZZ_N \otimes \wedge(x_1, \cdots, x_{n-2})$ in the zeroth and $(2n+1)$th columns which will become $K_0(SU(n+1),\delta)$ and $K_1(SU(n+1),\delta)$ respectively.

In order to conclude that the same is true for higher twisted $K$-theory, we see that the $E_2$-term of the Segal spectral sequence for higher twisted $K$-theory consists only of finite torsion groups, and even though we do not have information about the differentials in this sequence we may conclude that the $E_{\infty}$-term will also consist only of finite torsion groups and thus the limit of the sequence is a direct sum of torsion groups. Therefore we are in the setting of Proposition \ref{K-homology}, and we may use this to obtain the result for higher twisted $K$-theory from the computation for higher twisted $K$-homology.
\end{proof}

As we have not developed a simple interpretation of the Hurewicz map, we are forced to consider only the case in which this map is necessarily zero for other reasons. If the Hurewicz map could be better understood then this would likely allow for this result to be strengthened, leading to descriptions of the higher twisted $K$-groups when $\gcd(N,n!) \neq 1$. In the classical case, the torsion in these groups becomes very complicated and so it would be a great achievement to obtain explicit expressions for the torsion in the higher twisted $K$-theory groups of $SU(n)$ in general.

We also have a structural theorem which is applicable in a more general setting, but which provides slightly less information about the higher twisted $K$-theory groups.

\begin{theorem}\label{sun}
If $\delta_{k} \in H^{k}(SU(n),\ZZ)$ is given by $N \neq 0$ times any primitive generator of $H^*(SU(n),\ZZ)$ (all of which have odd degree) then $K^*(SU(n),\delta_k)$ is a finite abelian group and all elements have order a divisor of a power of $N$.
\end{theorem}

\begin{proof}
Again we proceed by induction on $n$, and observe that the base case has already been proved in Lemma \ref{SU(n)}. Hence we need only show that under this assumption it is true for $\delta_{2n-1} \in H^{2n-1}(SU(n+1),\ZZ)$. To do this, we once again use the classical fibration over $S^{2n+1}$ and apply the Segal spectral sequence, this time simply in higher twisted $K$-theory, to obtain 
\[E_2^{p,q} = H^p(S^{2n+1},K^q(SU(n),\delta_{2n-1})) \Rightarrow K^*(SU(n+1),\delta_{2n-1}).\]
Here we are again identifying $\iota^* \delta$ with $\delta$ using the isomorphism induced by inclusion on cohomology. Since $K^q(SU(n),\delta_{2n-1})$ is torsion with all elements having order a divisor of a power of $N$ by the inductive assumption, the same is true for $E_2$ and thus $E_{\infty}$. Finally, even if there are non-trivial extension problems to solve in order to obtain $K^*(SU(n+1),\delta_{2n-1})$, the result is still true as argued in the proof of Lemma \ref{SU(n)}.
\end{proof}

In some cases, this result yields particularly useful information. For instance, if $N=1$ then $K^*(X,\delta)$ vanishes identically and if $N = p^r$ is a prime power then $K^*(X,\delta)$ is a $p$-primary torsion group.

Whilst these results are not completely general, they do provide some insight into the complicated behaviour of the higher twisted $K$-theory of $SU(n)$ and lay the foundation for further research into this area. In the classical setting, Rosenberg is able to draw more general conclusions by using a universal coefficient theorem of Khorami \cite{Khorami}. This universal coefficient theorem relies on techniques that we have not developed in the higher twisted $K$-theory setting, but given a greater understanding of the cohomology groups of $\Aut(\OK)$ it may be possible to obtain a universal coefficient theorem in higher twisted $K$-theory which could allow for the results in this section to be generalised. Note that the approaches used here can also be applied to other compact, simply connected, simple Lie groups such as the symplectic groups $Sp(n)$ and $G_2$, and similar approaches exist for some non-simply connected groups such as the projective special unitary groups as used in \cite{MRosenberg}.

\subsection{Real projective spaces}

We now consider spaces with torsion in their cohomology, and show how in these cases the twists of $K$-theory can still be identified with odd-degree integral cohomology. We begin with odd-dimensional real projective space $\RR P^{2n+1}$, which has integral cohomology groups
\begin{equation}\label{RPnc}
H^p(\RR P^{2n+1}, \ZZ) = \begin{cases} \ZZ &\text{ if } p=0, 2n+1; \\ \ZZ_2 &\text{ if } 0 < p < 2n+1 \text{ is even}; \\ 0 &\text{ else}; \end{cases}
\end{equation}
and by the universal coefficient theorem the $\ZZ_2$-cohomology is
\[ H^p(\RR P^{2n+1}, \ZZ_2) = \begin{cases} \ZZ_2 &\text{ if } 0 \leq p \leq 2n+1; \\ 0 &\text{ else}. \end{cases} \]
Recall from the discussion preceding Theorem \ref{cohomology} that the full set of twists of $K$-theory over a space $X$ is given by the first group in a generalised cohomology theory $E_{\cO_{\infty}}^1(X)$, which is computed via a spectral sequence. In order to determine whether we may identify the twists of $K$-theory over $\RR P^{2n+1}$ with all odd-degree cohomology classes, we use this spectral sequence to compute $E_{\cO_{\infty}}^1(\RR P^{2n+1})$. The $E_2$-term of this spectral sequence is as follows.
\begin{center}
\begin{tikzpicture}
  \matrix (m) [matrix of math nodes,
    nodes in empty cells,nodes={minimum width=2.3ex,
    minimum height=5ex,outer sep=-5pt},
    column sep=1ex,row sep=1ex]{
    				& &  0  &  1  &  2 & 3 & 4 & \cdots & 2n & 2n+1 &\\
		0    	& &	\ZZ_2  & \circles \ZZ_2 & \ZZ_2 & \ZZ_2 & \ZZ_2 & \cdots & \ZZ_2 & \ZZ_2 & \\
		-1     	& &	0 &  0  & \circles 0 & 0 & 0 & \cdots & 0 & 0 & \\
		-2    	& &	\ZZ  & 0 & \ZZ_2 & \circles 0 & \ZZ_2 & \cdots & \ZZ_2 & \ZZ & \\
		-3    	& &	0 &  0  & 0 & 0 & \circles 0 & \cdots & 0 & 0 & \\
		\vdots & & \vdots & \vdots & \vdots & \vdots & \vdots & \circles \ddots & \vdots & \vdots & \\
		-2n+1	& &	0 &  0  & 0 & 0 & 0 & \cdots & \circles 0 & 0 & \\
		-2n	& & \ZZ  & 0 & \ZZ_2 & 0 & \ZZ_2 & \cdots & \ZZ_2 & \circles \ZZ & \\};
\draw[thick] (m-1-1.south) -- (m-1-11.south) ;
\draw[thick] (m-1-2.north) -- (m-8-2.south) ;
\end{tikzpicture}
\end{center}
The only differentials that could potentially be non-zero in this spectral sequence are maps $H^{odd}(\RR P^{2n+1},\ZZ_2) \to H^{even}(\RR P^{2n+1}, \ZZ)$ and $H^{even}(\RR P^{2n+1},\ZZ) \to H^{2n+1}(\RR P^{2n+1}, \ZZ)$. The latter of these will necessarily be zero since the target is torsion-free. Now, since the twists are determined specifically by the first group in this generalised cohomology theory, the only groups of interest in this spectral sequence are those circled and thus the only differentials that may have an effect are those from $H^1(\RR P^{2n+1},\ZZ_2)$. It is known, however, that the classical twists of $K$-theory are those coming from $H^1(X,\ZZ_2)$ and $H^3(X,\ZZ)$, and as such these groups always form a subgroup of $E^1_{\cO_{\infty}}(X)$. This means that the differentials leaving $H^1(\RR P^{2n+1},\ZZ_2)$ are all trivial, and thus $E^1_{\cO_{\infty}}(\RR P^{2n+1}) \cong \ZZ_2 \oplus \ZZ$. As the $H^1$ twists are studied in the classical case, the higher twists of interest are those coming from the $H^{2n+1}(\RR P^{2n+1},\ZZ) \cong \ZZ$ factor. The same argument may be applied to even-dimensional real projective space $\RR P^{2n}$, but this has no non-trivial odd-dimensional cohomology groups and so the only twists of $K$-theory are those coming from $H^1(\RR P^{2n},\ZZ_2) \cong \ZZ_2$.

We will now compute the higher twisted $K$-theory of $\RR P^{2n+1}$ for a $(2n+1)$-twist $\delta$. We provide two different proofs of this proposition using the Mayer--Vietoris sequence, the latter of which can be generalised to lens spaces.

\begin{proposition}\label{RPn}
Let $\delta \in H^{2n+1}(\RR P^{2n+1},\ZZ)$ be a twist of $K$-theory for $\RR P^{2n+1}$ which is $N \neq 0$ times the generator, with $n \geq 1$. The higher twisted $K$-theory of $\RR P^{2n+1}$ is then
\[ K^0(\RR P^{2n+1}, \delta) = \ZZ_{2^n} \quad \text{and} \quad K^1(\RR P^{2n+1},\delta) = \ZZ_N. \]
\end{proposition}

\begin{proof}[Proof 1]
Letting $\cA_{\delta}$ be the algebra bundle with fibre $\OK$ over $\RR P^{2n+1}$ corresponding to the twist $\delta$, we consider the short exact sequence of $C^*$-algebras
\[ 0 \to C_0(\RR^{2n+1}, \OK) \xrightarrow{\iota} C(\RR P^{2n+1}, \cA_{\delta}) \xrightarrow{\pi} C(\RR P^{2n}, \OK) \to 0, \]
where the maps $\iota$ and $\pi$ come from viewing $\RR^{2n+1}$ as the quotient of $\RR P^{2n+1}$ by $\RR P^{2n}$. The corresponding six-term exact sequence is 
\[
\begin{tikzcd}
K^0(\RR^{2n+1}) \ar{r}{\iota_*} &K^0(\RR P^{2n+1},\delta) \ar{r}{\pi_*} &K^0(\RR P^{2n}) \ar{d}{\partial} \\
K^1(\RR P^{2n}) \ar{u}{\partial} &K^1(\RR P^{2n+1},\delta) \ar{l}{\pi_*} &K^1(\RR^{2n+1}), \ar{l}{\iota_*}
\end{tikzcd}
\]
where higher twisted $K$-theory groups have been identified with topological $K$-theory groups via trivialisations, and this simplifies to
\[ 0 \to K^0(\RR P^{2n+1},\delta) \xrightarrow{\pi_*} \ZZ_{2^n} \oplus \ZZ \xrightarrow{\partial} \ZZ \xrightarrow{\iota_*} K^1(\RR P^{2n+1},\delta) \to 0. \]
In this case, the connecting map is given by $\partial(m,n) = nN$ as in 8.3 of \cite{BCMMS}, and this has kernel $\ZZ_{2^n}$ and cokernel $\ZZ_N$ when $N \neq 0$ as required.
\end{proof}

Although we have not discussed equivariant higher twisted $K$-theory, the definitions and properties from the classical case generalise immediately and we use this in the following proof.

\begin{proof}[Proof 2]
Viewing $\RR P^{2n+1}$ as the quotient $S^{2n+1} / \ZZ_2$, we take the short exact sequence
\[ 0 \to C_0(\RR \times S^{2n}, \OK) \xrightarrow{\iota} C(S^{2n+1},\cA_{\delta}) \xrightarrow{\pi} C(\{x_0, x_1\},\OK) \to 0 \]
where $S^{2n+1} \setminus \{x_0,x_1\} \cong \RR \times S^{2n}$ and $x_0$ and $x_1$ are related by the $\ZZ_2$-action. The associated six-term exact sequence in $\ZZ_2$-equivariant $K$-theory is
\[
\begin{tikzcd}
K^0_{\ZZ_2}(\RR \times S^{2n}) \ar{r}{\iota_*} &K^0_{\ZZ_2}(S^{2n+1},\delta) \ar{r}{\pi_*} &K^0_{\ZZ_2}(\{x_0,x_1\}) \ar{d}{\partial} \\
K^1_{\ZZ_2}(\{x_0,x_1\}) \ar{u}{\partial} &K^1_{\ZZ_2}(S^{2n+1},\delta) \ar{l}{\pi_*} &K^1_{\ZZ_2}(\RR \times S^{2n}), \ar{l}{\iota_*}
\end{tikzcd}
\]
where trivialisations have been used to identify higher twisted $K$-theory groups with topological $K$-theory groups. Here, since $\ZZ_2$ acts freely on $\{x_0, x_1\}$ then $K^*_{\ZZ_2}(\{x_0,x_1\}) = K^*(\{x_0\})$. Furthermore, since $\ZZ_2$ acts freely on $S^{2n+1}$ we also have $K^*_{\ZZ_2}(S^{2n+1},\delta) \cong K^*(\RR P^{2n+1},\delta)$, where the class $\delta$ on $\RR P^{2n+1}$ is identified with a class on $S^{2n+1}$ using the isomorphism induced by the projection map. In order to determine the equivariant $K$-theory of $\RR \times S^{2n}$, the untwisted version of this six-term exact sequence may be applied analogously to \cite{BCMMS} to find that $K^0_{\ZZ_2}(\RR \times S^{2n}) = \ZZ_{2^n}$ while $K^1_{\ZZ_2}(\RR \times S^{2n}) = \ZZ$. Thus this sequence reduces to
\[ 0 \to \ZZ_{2^n} \xrightarrow{\iota_*} K^0(\RR P^{2n+1},\delta) \xrightarrow{\pi_*} \ZZ \xrightarrow{\partial} \ZZ \xrightarrow{\iota_*} K^1(\RR P^{2n+1},\delta) \to 0, \]
where the connecting map $\partial: \ZZ \to \ZZ$ is once again multiplication by $N$ as in \cite{BCMMS}. This allows us to conclude that $K^0(\RR P^{2n+1},\delta) \cong \ZZ_{2^n}$ and $K^1(\RR P^{2n+1},\delta) \cong \ZZ_N$ as required.
\end{proof}

Note that this computation agrees with the classical case of the 3-twisted $K$-theory of $\RR P^3$. Although the Atiyah--Hirzebruch spectral sequence proved useful in our previous computations, it is not so helpful in this case due to the torsion in the cohomology of $\RR P^{2n+1}$ and so we will not use it.

\subsection{Lens spaces}

The equivariant computation for $\RR P^{2n+1}$ generalises nicely to lens spaces. While lens spaces of the form $S^3 / \ZZ_p$ for $p$ prime are particularly common, there exists a notion of higher lens space $L(n,p) = S^{2n+1} / \ZZ_p$ where $\ZZ_p$ identified with the $p^{th}$ roots of unity in $\CC$ acts on $S^{2n+1} \subset \CC^{n+1}$ by scaling. In this notation, the lens space $S^3 / \ZZ_p$ is $L(1,p)$ and real projective space $\RR P^{2n+1}$ can be viewed as $L(n,2)$. Since these higher lens spaces are quotients of $S^{2n+1}$, they have non-trivial cohomology in degree $2n+1$ and thus it is natural to consider their higher twisted $K$-theory.

As we did for $\RR P^{2n+1}$, we must determine the group $E^1_{\cO_{\infty}}(L(n,p))$ to see whether the twists of $K$-theory for $L(n,p)$ may be identified with all odd-degree cohomology classes. To do so, we observe that the cohomology groups of $L(n,p)$ are the same as those of $\RR P^{2n+1}$ in (\ref{RPnc}), but with $\ZZ_2$ replaced with $\ZZ_p$. Without needing to determine $H^*(L(n,p),\ZZ_2)$ we then follow the same argument to conclude that $E^1_{\cO_{\infty}}(L(n,p)) \cong H^1(L(n,p),\ZZ_2) \oplus H^{2n+1}(L(n,p),\ZZ)$. Thus we can view the higher twists of $K$-theory over $L(n,p)$ as integral $(2n+1)$-classes.

\begin{proposition}\label{lens}
Let $\delta \in H^{2n+1}(L(n,p),\ZZ)$ be a twist of $K$-theory for $L(n,p)$ which is $N \neq 0$ times the generator, where $n \geq 1$. The higher twisted $K$-theory of $L(n,p)$ is then
\[ K^0(L(n,p), \delta) = \ZZ_{p^n} \quad \text{and} \quad K^1(L(n,p),\delta) = \ZZ_N. \]
\end{proposition}

\begin{proof}
Letting $A \subset L(n,p)$ consist of a $\ZZ_p$-orbit and $\cA_{\delta}$ denote the algebra bundle with fibre $\OK$ over $L(n,p)$ corresponding to the twist $\delta$, we take the short exact sequence
\[ 0 \to C_0(S^{2n+1} \setminus A, \OK) \xrightarrow{\iota} C(S^{2n+1},\cA_{\delta}) \xrightarrow{\pi} C(A,\OK) \to 0. \]
The associated six-term exact sequence in $\ZZ_p$-equivariant higher twisted $K$-theory is
\[
\begin{tikzcd}
K^0_{\ZZ_p}(S^{2n+1} \setminus A) \ar{r}{\iota_*} &K^0_{\ZZ_p}(S^{2n+1},\delta) \ar{r}{\pi_*} &K^0_{\ZZ_p}(A) \ar{d}{\partial} \\
K^1_{\ZZ_p}(A) \ar{u}{\partial} &K^1_{\ZZ_p}(S^{2n+1},\delta) \ar{l}{\pi_*} &K^1_{\ZZ_p}(S^{2n+1} \setminus A), \ar{l}{\iota_*}
\end{tikzcd}
\]
where higher twisted $K$-theory groups have been identified with topological $K$-theory groups via trivialisations. Since $\ZZ_p$ acts freely on the compact set $A$ we have $K^*_{\ZZ_p}(A) = K^*(\{x_0\})$, and similarly $K^*_{\ZZ_p}(S^{2n+1},\delta) = K^*(L(n,p),\delta)$ where once again $\delta \in H^{2n+1}(L(n,p),\ZZ)$ is identified with $\delta \in H^{2n+1}(S^{2n+1},\ZZ)$ via the isomorphism induced by the projection map. It remains to determine $K^*_{\ZZ_p}(S^{2n+1} \setminus A)$. To do so, we require some basic properties of equivariant $K$-theory which carry over from the standard case. Firstly, identifying $K^*(S^{2n+1},A)$ with $K^*(S^{2n+1} \setminus A)$, we have the short exact sequences
\[ 0 \to K^m_{\ZZ_p}(S^{2n+1} \setminus A) \to K^m_{\ZZ_p}(S^{2n+1}) \to K^m_{\ZZ_p}(A) \to 0 \]
for $m = 0$ and $1$. Since $K^*_{\ZZ_p}(A) = K^*(\{x_0\})$, this implies that $K^n_{\ZZ_p}(S^{2n+1} \setminus A) \cong \widetilde{K}^n_{\ZZ_p}(S^{2n+1})$, which is simply isomorphic to $\widetilde{K}^n(L(n,p))$. By a computation analogous to that for $\RR P^{2n+1}$ using Corollary 2.7.6 of \cite{AtiyahK}, we observe that $K^0(L(n,p)) \cong \ZZ_{p^n} \oplus \ZZ$ while $K^1(L(n,p)) \cong \ZZ$. Thus the sequence reduces to
\[ 0 \to \ZZ_{p^n} \xrightarrow{\iota_*} K^0(L(n,p),\delta) \xrightarrow{\pi_*} \ZZ \xrightarrow{\partial} \ZZ \xrightarrow{\iota_*} K^1(L(n,p),\delta) \to 0, \]
where the connecting map $\partial: \ZZ \to \ZZ$ is multiplication by $N$ as in 8.3 of \cite{BCMMS}. Hence $K^0(L(n,p),\delta) \cong \ZZ_{p^n}$ and $K^1(L(n,p),\delta) \cong \ZZ_N$ as required.
\end{proof}

This provides a generalisation of the result for $\RR P^{2n+1}$ in Proposition \ref{RPn} as well as for the 3-dimensional lens spaces in Section 8.4 of \cite{BCMMS}.

\subsection{\texorpdfstring{$SU(2)$}{SU(2)}-bundles}

For the final computation we turn to a class of examples of greater relevance in physics. It is of great interest in both string theory and mathematical gauge theory to investigate $SU(2)$-bundles over manifolds $M$. The link with string theory appears through spherical T-duality; a generalisation of T-duality investigated in a series of papers by Bouwknegt, Evslin and Mathai \cite{BEM1,BEM2,BEM3}. As shown by the authors, this form of duality provides a map between certain conserved charges in type IIB supergravity and string compactifications. They also find links between spherical T-duality and higher twisted $K$-theory. Specifically in the case that $M$ is a compact oriented 4-manifold, the authors compute the 7-twisted $K$-theory of the total space of a principal $SU(2)$-bundle over $M$ in terms of its second Chern class up to an extension problem. The ordinary 3-twisted $K$-theory can be computed using standard techniques, which leaves the 5-twisted $K$-theory to be computed. In our computations we will assume that $M$ has torsion-free cohomology, which is true if the additional assumption that $M$ is simply connected is made but this is not necessary.

Note that we are not requiring the $SU(2)$-bundle $P$ over $M$ to be a principal bundle. The computation will be valid for both principal $SU(2)$-bundles as used in \cite{BEM1} as well as oriented non-principal $SU(2)$-bundles used in \cite{BEM2}, the former of which correspond to unit sphere bundles of quaternionic line bundles and the latter of which correspond to unit sphere bundles of rank 4 oriented real Riemannian vector bundles. The reason that we need not distinguish between these types of bundles is that there is a Gysin sequence in each case allowing the cohomology of $P$ to be computed. Given an $SU(2)$-bundle of either of these forms $\pi: P \to M$, there is a Gysin sequence of the form
\[ \cdots \to H^k(M,\ZZ) \xrightarrow{\pi^*} H^k(P,\ZZ) \xrightarrow{\pi_*} H^{k-3}(M,\ZZ) \xrightarrow{\cup e(P)} H^{k+1}(M,\ZZ) \to \cdots \]
where $e(P)$ denotes the Euler class of $P$ which may be identified with the second Chern class of the associated vector bundle in the principal bundle case. In this sequence, the pushforward $\pi_*$ is defined by using Poincar\'e duality to change from cohomology to homology, using the pushforward in homology and once again employing Poincar\'e duality to switch back, which explains the degree shift. We can use this to compute the integral cohomology of $P$ in terms of that of $M$. In what follows, we will assume that all exact sequences split in order to compute the higher twisted $K$-theory up to extension problems. Although this will not be true in all cases, it still yields meaningful results as in \cite{BEM1,BEM2}.

Firstly, we assume that $e(P) = 0$ in which case the obstruction to $\pi: P \to M$ having a section vanishes and therefore the Gysin sequence truly does split at $\pi_*$. This yields
\[ H^k(P,\ZZ) \cong H^k(M,\ZZ) \oplus H^{k-3}(M,\ZZ). \]
Thus we obtain
\[ H^k(P,\ZZ) \cong
\begin{cases}
\ZZ &\text{ if } k=0, 7; \\
H^k(M,\ZZ) &\text{ if } k=1, 2; \\
H^3(M,\ZZ) \oplus \ZZ &\text{ if } k=3; \\
H^1(M,\ZZ) \oplus \ZZ &\text{ if } k=4; \\
H^{k-3}(M,\ZZ) &\text{ if } k=5, 6.
\end{cases}
\]
If $e(P) = j \in \ZZ$ with $j \neq 0$ on the other hand, then the cup product with $e(P)$ from $H^0(M,\ZZ)$ to $H^4(M,\ZZ)$ will be multiplication by $j$. We obtain the same integral cohomology of $P$ as above in all degrees except $3$ and $4$, and to compute these we use the Gysin sequence as follows:
{ \small \[ 0 \to H^3(M,\ZZ) \xrightarrow{\pi^*} H^3(P,\ZZ) \xrightarrow{\pi_*} H^0(M,\ZZ) \xrightarrow{\cup e(P)} H^4(M,\ZZ) \xrightarrow{\pi^*} H^4(P,\ZZ) \xrightarrow{\pi_*} H^1(M,\ZZ) \to 0. \]}
Since the cup product here has trivial kernel, we conclude that $H^3(P,\ZZ) \cong H^3(M,\ZZ)$. Similarly, the cokernel is $\ZZ_j$ and hence $H^4(P,\ZZ) \cong H^1(M,\ZZ) \oplus \ZZ_j$ assuming that the sequence is split at $\pi_*: H^4(P,\ZZ) \to H^1(M,\ZZ)$.

In the case that $e(P) = 0$, we see that $P$ has torsion-free cohomology and thus the 5-twists given by $H^5(P,\ZZ) \cong H^2(M,\ZZ)$ can be considered. If $e(P) = j \neq 0$, however, then there is torsion in the cohomology of $P$ and so it must be determined whether all of the elements of $H^5(P,\ZZ)$ correspond to twists. The $E_2$-term of the Atiyah--Hirzebruch spectral sequence to compute $E^1_{\cO_{\infty}}(P)$ is shown below.
{\footnotesize \begin{center}
\begin{tikzpicture}
  \matrix (m) [matrix of math nodes,
    nodes in empty cells,nodes={minimum width=2.3ex,
    minimum height=5ex,outer sep=-5pt},
    column sep=1ex,row sep=1ex]{
    				& &  0  &  1  &  2 & 3 & 4 & 5 & 6 & 7 &\\
		0    	& &	\ZZ_2  & H^1(P,\ZZ_2) & H^2(P,\ZZ_2) & H^3(P,\ZZ_2) & H^4(P,\ZZ_2)  & H^5(P,\ZZ_2) & H^6(P,\ZZ_2) & \ZZ_2 & \\
		-1     	& &	0 & 0 & 0 & 0 & 0 & 0 & 0 & 0 & \\
		-2    	& &	\ZZ  & H^1(M,\ZZ) & H^2(M,\ZZ) & H^3(M,\ZZ) & H^1(M,\ZZ) \oplus \ZZ_j  & H^2(M,\ZZ) & H^3(M,\ZZ) & \ZZ & \\
		-3    	& &	0 & 0 & 0 & 0 & 0 & 0 & 0 & 0 & \\
		-4    	& &	\ZZ  & H^1(M,\ZZ) & H^2(M,\ZZ) & H^3(M,\ZZ) & H^1(M,\ZZ) \oplus \ZZ_j  & H^2(M,\ZZ) & H^3(M,\ZZ) & \ZZ & \\};
\draw[thick] (m-1-1.south) -- (m-1-11.south) ;
\draw[thick] (m-1-2.north) -- (m-6-2.south) ;
\end{tikzpicture}
\end{center}}

Although there is torsion in $H^4(P,\ZZ)$ which will affect the computation of $E^*_{\cO_{\infty}}(P)$, this will not have an effect on $E^1_{\cO_{\infty}}(P)$ much like was the case for $\RR P^{2n+1}$. Since the differentials are torsion operators then $d_3: H^2(M,\ZZ) \to H^2(M,\ZZ)$ will be zero, and as $H^1(P,\ZZ_2)$ makes up a subset of the twists then the differential $d_3: H^1(P,\ZZ_2) \to H^4(P,\ZZ)$ will also necessarily be zero. Thus we may conclude that the twists of $K$-theory are given by odd-degree cohomology in this case, and so it is sensible to use 5-twists and 7-twists for $P$.

To compute the higher twisted $K$-theory groups themselves, the twisted Atiyah--Hirzebruch spectral sequence may be used. The $E_2$-term for $e(P) = 0$ is as follows.
{\footnotesize \begin{center}
\begin{tikzpicture}
  \matrix (m) [matrix of math nodes,
    nodes in empty cells,nodes={minimum width=2.3ex,
    minimum height=5ex,outer sep=-5pt},
    column sep=1ex,row sep=1ex]{
		2    	& &	\ZZ  & H^1(M,\ZZ) & H^2(M,\ZZ) & H^3(M,\ZZ) \oplus \ZZ & H^1(M,\ZZ) \oplus \ZZ  & H^2(M,\ZZ) & H^3(M,\ZZ) & \ZZ & \\
		1     	& &   0 &  0  & 0 & 0 & 0 & 0 & 0 & 0 & \\
		0    	& &	\ZZ  & H^1(M,\ZZ) & H^2(M,\ZZ) & H^3(M,\ZZ) \oplus \ZZ & H^1(M,\ZZ) \oplus \ZZ  & H^2(M,\ZZ) & H^3(M,\ZZ) & \ZZ & \\
     \strut & &  0  &  1  &  2 & 3 & 4 & 5 & 6 & 7 & \strut \\};
\draw[thick] (m-1-2.north) -- (m-4-2.south) ;
\draw[thick] (m-4-1.north) -- (m-4-11.north) ;
\end{tikzpicture}
\end{center}}

Similarly, the $E_2$-term for $e(P) = j \neq 0$ is below.

{\footnotesize \begin{center}
\begin{tikzpicture}
  \matrix (m) [matrix of math nodes,
    nodes in empty cells,nodes={minimum width=2.3ex,
    minimum height=5ex,outer sep=-5pt},
    column sep=1ex,row sep=1ex]{
		2    	& &	\ZZ  & H^1(M,\ZZ) & H^2(M,\ZZ) & H^3(M,\ZZ) & H^1(M,\ZZ) \oplus \ZZ_j  & H^2(M,\ZZ) & H^3(M,\ZZ) & \ZZ & \\
		1     	& &   0 &  0  & 0 & 0 & 0 & 0 & 0 & 0 & \\
		0    	& &	\ZZ  & H^1(M,\ZZ) & H^2(M,\ZZ) & H^3(M,\ZZ) & H^1(M,\ZZ) \oplus \ZZ_j  & H^2(M,\ZZ) & H^3(M,\ZZ) & \ZZ & \\
   \strut  & &  0  &  1  &  2 & 3 & 4 & 5 & 6 & 7 & \strut \\};
\draw[thick] (m-1-2.north) -- (m-4-2.south) ;
\draw[thick] (m-4-1.north) -- (m-4-11.north) ;
\end{tikzpicture}
\end{center}}

In both cases, we may argue that the $d_3$ differential $Sq^3$ is zero as follows. Firstly, $Sq^3$ of course annihilates $H^k(P,\ZZ)$ for $0 \leq k \leq 2$ by Theorem 4L.12 of \cite{HatcherAT} as well as $5 \leq k \leq 7$ since the image of the map is a trivial cohomology group, which leaves only $k=3$ and $k=4$ to be considered. But the image of $Sq^3$ is a $\ZZ_2$-torsion element by definition, and there is no torsion in $H^6(P,\ZZ)$ or $H^7(P,\ZZ)$ and hence this differential must be zero. Similarly, the torsion part of $d_5$ will annihilate all cohomology classes, leaving only the cup product with the twisting class $-\delta \in H^5(P,\ZZ)$ to be considered.

To determine how cup product with the twisting class affects the cohomology, the isomorphisms between the cohomology groups of $M$ and $P$ need to be viewed more explicitly. Observe that $\pi_*: H^k(P,\ZZ) \to H^{k-3}(M,\ZZ)$ is an isomorphism for $k=5, 6, 7$ and similarly $\pi^*: H^k(M,\ZZ) \to H^k(P,\ZZ)$ is an isomorphism for $k = 0, 1, 2$. Fixing a twist $\delta \in H^5(P,\ZZ)$, there is an $\eta \in H^2(M,\ZZ)$ such that $(\pi_*)^{-1}(\eta) = \delta$. Then taking $\alpha_k \in H^k(M,\ZZ)$ for $k=0, 1, 2$ so that $\pi^*(\alpha_k) \in H^k(P,\ZZ)$, we see that $\pi_*(\delta \cup \pi^*(\alpha_k)) \in H^k(M,\ZZ)$ will be the image of the cup product in the cohomology of $M$. But by a property of pushforwards, this is equal to $\pi_*(\delta) \cup \alpha_k = \eta \cup \alpha_k$. Hence the cup product with $\delta$ on the cohomology of $P$ is equivalent to cup product with  $\eta$ on the cohomology of $M$. Since this cup product is injective on $H^0(P,\ZZ)$, all higher differentials will be zero and thus the $E_{\infty}$-terms of the spectral sequences can be determined. We will not explicitly write the $E_{\infty}$-terms due to their size, but they can easily be determined from above.

Finally we may conclude that, up to extension problems, the 5-twisted $K$-theory of $P$ when $e(P) = 0$ is
\begin{align*}
K^0(P,\delta) &= \ker{ \cup \eta|_{H^2(M,\ZZ)}} \oplus H^1(M,\ZZ) \oplus \ZZ \oplus \coker{ \cup \eta|_{H^1(M,\ZZ)}}; \\
K^1(P,\delta) &= \ker{ \cup \eta|_{H^1(M,\ZZ)}} \oplus H^3(M,\ZZ) \oplus \ZZ \oplus H^2(M,\ZZ) / |\eta| \ZZ \oplus \coker{ \cup \eta|_{H^2(M,\ZZ)}};
\end{align*}
and when $e(P) = j \neq 0$ is
\begin{align*}
K^0(P,\delta) &= \ker{ \cup \eta|_{H^2(M,\ZZ)}} \oplus H^1(M,\ZZ) \oplus \ZZ_j \oplus \coker{ \cup \eta|_{H^1(M,\ZZ)}}; \\
K^1(P,\delta) &= \ker{ \cup \eta|_{H^1(M,\ZZ)}} \oplus H^3(M,\ZZ) \oplus H^2(M,\ZZ) / |\eta| \ZZ \oplus \coker{ \cup \eta|_{H^2(M,\ZZ)}}.
\end{align*}

To be more explicit about the extension problems involved, we illustrate how these higher twisted $K$-theory groups may differ if the extension problems are non-trivial.

Considering the case of $e(P) = j \neq 0$ above, determining the even-degree higher twisted $K$-theory group of $P$ would require solving the following:
\begin{align*}
&0 \to H^1(M,\ZZ) \oplus \ZZ_j \to A \to \coker{ \cup \eta|_{H^1(M,\ZZ)}} \to 0; \\
&0 \to \ker{ \cup \eta|_{H^2(M,\ZZ)}} \to K^0(P,\delta) \to A \to 0.
\end{align*}

These extension problems should be kept in mind when viewing the direct sums in the expressions above.

Of course these higher twisted $K$-theory groups are heavily dependent on the ring structure of the cohomology of $M$ as is evident by the presence of the kernel and cokernel of the cup product on $M$ in the expressions given above, but given a specific 4-manifold $M$ with torsion-free cohomology, e.g.\ a simply connected 4-manifold satisfying the previous assumptions, this determines the 5-twisted $K$-theory of $P$ up to extension problems.

\begin{example}
We apply the formulas given above to the manifold $M = S^2 \times S^2$. This space has trivial cohomology in degrees 1 and 3, and $H^2(M,\ZZ) \cong \ZZ \oplus \ZZ$. To specify the ring structure on cohomology, the cup product of the two generators of $H^2(M,\ZZ)$ is the generator of $H^4(M,\ZZ)$.

Consider a twist $\delta \in H^2(M,\ZZ)$ given by $\delta = (L \alpha_0, N \beta_0)$ with $\alpha_0, \beta_0$ the generators and $L, N \neq 0$. Then the cup product map $H^2(M,\ZZ) \to H^4(M,\ZZ)$ is $(a,b) \mapsto Lb + Na$, which has kernel $\{(Li/k, Ni/k) : i \in \ZZ\} \cong \ZZ$ and cokernel $\ZZ_k$ where $k$ denotes the greatest common divisor of $L$ and $N$. From this, we see that when $e(P) = 0$ we have
\begin{align*}
K^0(P,\delta) &= \ZZ \oplus \ZZ; \\
K^1(P,\delta) &= \ZZ \oplus \ZZ_L \oplus \ZZ_N \oplus \ZZ_k;
\end{align*}
and when $e(P) = j \neq 0$ the groups are
\begin{align*}
K^0(P,\delta) &= \ZZ \oplus \ZZ_j; \\
K^1(P,\delta) &= \ZZ_L \oplus \ZZ_N \oplus \ZZ_k.
\end{align*}
\end{example}

Since there is no specified 5-class in the setting of spherical T-duality, the physical interpretation of these 5-twisted $K$-theory groups is not as clear as in the case of 7-twists. Nevertheless, the work of Bouwknegt, Evslin and Mathai provides a link between spherical T-duality and supergravity theories in Type IIB string theory, and so further research into this area may shed light on the physical meaning of the computations that we have performed. Tying these computations together with physics may then provide further insight into certain aspects of string theory.

To conclude, we remark that many of the computations presented in this section can be used as the starting point for further investigation, particularly in obtaining more general results for the higher twisted $K$-theory of Lie groups and in exploring the link between higher twisted $K$-theory and string theory.


\bibliographystyle{plain}

\end{document}